\documentclass[12pt]{article}%
\usepackage{geometry}
\usepackage{amsmath}
\usepackage{amsfonts}
\usepackage{amssymb}
\usepackage{graphicx}%
\usepackage{float}
\newcommand*{\qed}{\hfill\ensuremath{\blacksquare}}%
\usepackage{mathabx}
\usepackage{xcolor}
\usepackage{mathtools}
\usepackage{mathrsfs}

\usepackage[colorlinks=true,linkcolor=blue,citecolor=magenta]{hyperref}

\providecommand{\U}[1]{\protect\rule{.1in}{.1in}}

\newtheorem{theorem}{Theorem}[section]

\newtheorem{assumption}[theorem]{Assumption}

\newtheorem{corollary}[theorem]{Corollary}

\newtheorem{definition}[theorem]{Definition}

\newtheorem{lemma}[theorem]{Lemma}

\newtheorem{proposition}[theorem]{Proposition}
\newtheorem{remark}[theorem]{Remark}

\newenvironment{proof}[1][Proof]{\noindent\textbf{#1.} }{\ \rule{0.5em}{0.5em}}

\renewcommand{\tilde}{\widetilde}

\renewcommand {\epsilon}{\varepsilon}

\newcommand{\EE}{\mathbb{E}}

\newcommand{\HH}{\mathbb{H}}

\newcommand{\LL}{\mathbb{L}}

\newcommand{\NN}{\mathbb{N}}

\newcommand{\PP}{\mathbb{P}}

\newcommand{\RR}{\mathbb{R}}

\newcommand{\dd}{\mathrm{d}}
\newcommand{\cL}{\mathcal{L}}

\newcommand{\fF}{\mathcal{F}}

\newcommand{\cX}{\mathcal{X}}
\newcommand{\bF}{\mathbf{F}}
\newcommand{\bC}{\mathbf{C}}

\newcommand{\sca}{\,\begin{picture}(-1,1)(-1,-3)\circle*{2.5}\end{picture}\ }

\begin{document}

\title{Microscopic derivation of non-local models  with anomalous diffusions from stochastic  particle systems}
\author{Marielle Simon\thanks{Universite Claude Bernard Lyon 1, CNRS, Ecole Centrale de Lyon, INSA Lyon, Université Jean Monnet, ICJ UMR5208,
69622 Villeurbanne, France. Email: \textsl{msimon@math.univ-lyon1.fr} \textit{and} GSSI, L'Aquila, Italy.}, Christian
Olivera\thanks{Departamento de Matem\'{a}tica, Universidade Estadual de
Campinas, Brazil. Email: \textsl{colivera@ime.unicamp.br}.}}

\date{}
\maketitle

%\tableofcontents

\begin{abstract}
  This paper considers a large class of nonlinear integro-differential scalar equations which involve an  anomalous diffusion (\textit{e.g.}~driven by a fractional Laplacian) and a non-local singular convolution kernel. Each of those singular equations is obtained as the macroscopic limit of an interacting particle system  modeled as $N$ coupled stochastic differential equations driven by L\'{e}vy processes. In particular we derive quantitative estimates between  the  microscopic empirical measure of the particle system and the solution to the limit equation in some non-homogeneous  Sobolev space. 
 Our result only requires very weak regularity on the interaction kernel, therefore it includes numerous applications, \textit{e.g.}: the $2d$
turbulence model (including the quasi-geostrophic equation) in  sub-critical regime, the $2d$ generalized Navier-Stokes equation, the fractional Keller-Segel equation in any dimension, and the fractal Burgers equation. 
\end{abstract}


\vspace{0.3cm} \noindent {\bf MSC2010 subject classification:} 60H20, 60H10, 60F99, 35K55.

\section{Introduction\label{Intro}}

The  main goal of this paper is to provide  a microscopic derivation  of non-local partial differential equations
from interacting stochastic particle systems. Precisely,  we consider the following class of scalar equations, written on the form
\begin{equation}\label{eq:PDE}
    \partial_t u(t, x)-  \mathcal{L} u(t, x)+  \mathrm{div}( K(u) u   )(t,x)  = 0, \qquad (t,x) \in \RR_+\times \RR^d,
\end{equation}
where $u:\RR_+\times\RR\to\RR$ is a scalar field, and some initial data  $ u(0,\cdot)=  u_{0}(\cdot)$ is given.
Here, $\mathcal L$ is the non-local operator corresponding to a symmetric $\sigma$-stable L\'evy process with $\sigma \in (1,2)$ (as defined in \eqref{OpeL} below) and $K$ is a singular integral  operator which satisfies a few assumptions specified in Section \ref{sec:setting}. Remarkably, under those assumptions, equation \eqref{eq:PDE} covers numerous well-known examples. Let us cite for instance the  subcritical \textit{$\eta$-turbulence model} in $2d$ \cite{Iwa}, which governs some geo-physical fluids. More illustrations will be detailed in Section \ref{sec:ex}.

\bigskip

Our main contribution is to provide a rigorous microscopic derivation  of the class of  
singular non-local equations of the form \eqref{eq:PDE}, supplemented by quantitative convergence estimates.  More precisely, for any $N\in\mathbb{N}$, 
we  consider the  $N$-particle dynamics evolving according to a system  of coupled stochastic differential equations in
 $\mathbb{R}^{d}$ of the following general form: for any $i=1,\dots,N$, the position $X_t^{i,N} \in \RR^d$ of the $i$--th particle at time $t$ is ruled by
\begin{equation}
\dd X_{t}^{i,N}=   \bF\bigg(\frac{1}{N}\sum_{k=1}^{N}  K(V^{N})(X_{t}^{i,N} -X_{t}^{k,N})\bigg) \; \dd t + \dd L_{t}^{i}.
 \label{itoassS0}%
\end{equation}
where: 
\begin{itemize}
\item the family of processes $\{t\mapsto L_{t}^{i}\}_{i\in\mathbb{N}}$ is a family of independent $\mathbb{R}^{d}$-valued symmetric $\sigma$-stable   L\'{e}vy processes on a filtered probability space $\left(\Omega,\mathcal{F},(\mathcal{F}_{t})_{t\geqslant 0},\mathbb{P}\right)$, whose infinitesimal generator is given by $\mathcal{L}$ and $\sigma\in(1,2)$. As an example, one may think about a fractional diffusion generated by the  fractional Laplacian $\mathcal{L}=-(-\Delta)^{\frac{\sigma}{2}}$. Note that the case $\sigma=2$ would correspond to standard Brownian motion;

\item the \emph{interaction potential} $V^N: \RR^d\to \RR_+$ is continuous and can be written in the form $V^N(x)=N^{d\beta}V(N^{\beta}x)$ for some $\beta >0$ and some smooth probability density $V$ which is compactly supported. Therefore, the particles \textit{moderately} interact with each other: in other words, the particles $i$ and $j$ are interacting only when  $|X_{i}-X_{j} |\lesssim   N^{-\beta}$;
    
\item the function $\bF:\RR\to\RR$ is bounded, $C^2$ regular, and is very close to the identity function on a compact set (see  Section \ref{sec:setting} for a precise definition).
 \end{itemize} 
The microscopic \emph{empirical process} of this $N$-particle system, which is a probability measure on the ambient space $\RR^d$,  is given as usual by \begin{equation}\mu_{t}^{N}:=\frac{1}{N} \sum_{i=1}^{N}\delta_{X_{t}^{i,N}}, \qquad t\geqslant 0, \label{eq:mutN}\end{equation} where $\delta_{a}$ is the delta Dirac measure concentrated at $a \in \RR^d$. Then, $(\mu_{t}^{N})_{t\geqslant 0}$ is a  measure--valued process associated to the $\RR^d$--valued processes $\{t\mapsto X_{t}^{i,N}\}_{i=1,...,N}$.  

The main goal of this paper is to investigate the large $N$ limit of  the dynamical process
$(\mu_{t}^{N})_{t\geqslant 0}$.
For that purpose, we introduce the \emph{mollified empirical measure} \[u_t^{N}:= V^{N} \ast \mu_t^{N} = \int_{\RR^d} V^N(\cdot-y) \,\mu_t^N(\dd y),\] 
which is more regular than $\mu_t^N$. Theorem \ref{th:rate2sing} below provides  a quantitative estimate of the distance between the mollified measure $u_t^N$ 
and the unique solution  $u$ of  equation \eqref{eq:PDE}, which   takes the following form: 
there exist constants $c,C>0$ such that for any $N\in\mathbb{N}$,
\begin{align*} 
   \sup_{t\in [0,T]}\EE\Big[  \| u_t^N-u \|_{\tilde{\mathcal{X}}}\Big]
&\leq  C\EE \Big[\|u^N_0- u_0 \|_{\tilde{\mathcal{X}}}\Big]   + c N^{-\varrho},
\end{align*}
where $\varrho$ is an explicit positive parameter, $T>0$ is a time horizon, and $\tilde{\mathcal X}$ is some functional space, suitably chosen. In particular, the  left-hand side distance  vanishes as $N\to \infty$, provided that the latter is true at time $t=0$.  As a corollary we also obtain the convergence of the original empirical measure $\mu_t^N$ to the solution of \eqref{eq:PDE}, as stated in Corollary \ref{Colo}.
The main novelty of the present work is the derivation of quantitative convergence rates thanks to a suitable choice of the norm  measuring the distance between the approximation and the solution. In particular we are able to show convergence of the empirical  measure without using compactness arguments.

\bigskip
 
Along the proof, two cases are distinguished, depending on the assumptions made on the singular kernel $K$ in Assumption \ref{assump}: \begin{itemize} \item  in \textbf{Case 1} (which will include  as illustrating examples  the $2d$ $\eta$-turbulence models and the scalar Burgers equation), the space $\tilde{\mathcal{X}}$ corresponds to some Bessel potential space $\mathbb{H}_p^\alpha$ (see \eqref{defH}) endowed with its natural norm denoted below by $\|\cdot\|_{\alpha,p}$ ;
\item in \textbf{Case 2} (which will include the $2d$ Navier-Stokes equation and the Keller-Segel models), the space $\tilde{\mathcal{X}}$ becomes $\mathbb{H}_p^{\alpha} \cap \mathbb{L}^1$ endowed with a \emph{distorted norm} which reads as \[\|u\|_{\tilde{\mathcal{X}}}=\|u\|_{\alpha,p}+\|K(u)\|_{\mathbb{L}^p}.\]
\end{itemize}
 The proof takes some inspiration from the semi-group approach exposed for the first time  in \cite{flan}
and then used in a series of works \cite{flan4,flan3,flan6,  Simon-O, Toma}. We note that the genesis of the stochastic approximations introduced in this series comes from the original papers of Oelschl\"ager  \cite{Oel1} and Jourdain and Méléard \cite{Mela}, where moderate interacting particles are investigated. The starting point consists in applying Itô's formula to any observable of the form $\phi(  X_{t}^{i,N})$, for some class of  test functions $\phi$, and to obtain the equation satisfied by the mollified measure $u_t^N$ in \textit{mild form} (see \eqref{eq:mild2} below). Then, most of the argument  lies in  estimating the distance with the continuous solution by means of  semi-group estimates, leading to the necessary conditions to apply Gronwall's Lemmas.  In \textbf{Case 2}, the distorted norm comes out due to the weaker assumption made on the operator $K$, which prevents from directly closing the Gronwall inequalities. 

\subsection*{Related works}

The problem of deriving non-local models with anomalous diffusion from microscopic models of  stochastic interacting particles
has not been highly investigated in the literature. As we will see  in more details in Section \ref{sec:ex}, our result may be applied to the following well-known models: the $2d$ quasi-geostrophic equation SQG \cite{Const}; its generalized version, namely the $2d$ $\eta$-turbulence model \cite{Iwa}; the scalar fractal Burgers equation \cite{Biler}; the fractional parabolic-elliptic Keller-Segel model \cite{Biler2}; and the $2d$ Navier-Stokes equation \cite{WuNS}. Besides, up to our knowledge, there is no result providing quantitative  estimates in particle approximations of such equations, despite the great importance 
it represents both theoretically and in applications.

The works   \cite{Biler,Mela2} provide a particle approximation of PDEs of type \eqref{eq:PDE},
when $K(u)= k\ast u$  and $k(x)\approx \frac{1}{|x|^{d-\epsilon}}$,     with $0<\epsilon<d$. More precisely, in
\cite{Mela2} the authors construct a McKean-style non-linear process for the cutoff versions of this equation, 
and then use it in order to create an interacting particle system whose empirical measure  strongly converges
to the solution. In  \cite{Biler}  a weak convergence result of this type had been previously obtained.   However their study does not cover  operators  $K$   with singularity of order $-d$ (for instance $k(x)\approx \frac{1}{|x|^{d}}$ like in the  $2d$-SQG equation),  and besides, the distance between the approximation and the continuous solution has not been quantitatively estimated. Recently, \cite{flan5}
 and  \cite{Romi} have  obtained a (also non quantitative) microscopic derivation of  the inviscid and dissipative generalized surface quasi-geostrophic (SQG) equation.

	In the case of  Brownian diffusions (instead of Lévy processes) there are some recent works  on PDEs with singular kernels:  in \cite{Jabin}, the authors 
	work under some assumption which includes the Biot-Savart kernel, but  does not cover  the Coulomb and  Keller-Segel kernels ; in  \cite{Bres},  combining the relative entropy with a modulated energy introduced by  Serfaty in \cite{Serfaty},  the authors work with attractive, gradient-flow  kernels including the $2d$ Keller-Segel model.  For recent advances around this  approach, we refer for instance to \cite{Gui, Serfa}. Finally in \cite{Toma}  the authors work under very general  assumptions
on the kernel $K$ (including the Keller-Segel model in any dimension) but it does not cover the  singularity of order $-d$. Let us also mention the work \cite{BossyTalay} which investigates the Burgers equation, and \cite{MeleardNS2d} where the authors derive the $2d$ Navier-Stokes equation. Also, \cite{Carri} and \cite{Fournier}  obtain a rate of convergence for the intricate Landau equation. 
We finally refer to  \cite{Jabin2}  and references therein for a recent review on the mean-field limit problems for stochastic interacting particle systems.

\subsection*{Outline of the paper}

In Section \ref{sec:results} we define the microscopic dynamics, in particular we introduce all the quantities  which appear in the stochastic particle system \eqref{itoassS0}, and state our assumptions. The main results are stated in Theorem \ref{th:rate2sing} and Corollary \ref{Colo}, then we conclude Section \ref{sec:results} by listing several examples which are covered by our work. In Section \ref{sec:prelim} we recall some classical tools which will be needed in the proof, and write as a starting point the equation satisfied by the mollified empirical measure in mild form. Section \ref{sec:proof1} (resp.~Section \ref{sec:proof2}) is devoted to the proof of Theorem \ref{th:rate2sing} (resp.~Corollary \ref{Colo}).

\section{Main results and examples of applications}
\label{sec:results} 
%This section is organised as follows: in Section \ref{sec:notation} we first introduce a few notations. All characteristics of the models and their assumptions are then presented in Section \ref{sec:setting}. Note that two cases  will be distinguished according to the regularity of the linear operator $K$. Our main theorems are presented in Section \ref{sec:result}. Finally,  in Section \ref{sec:ex} we list several examples satisfying our assumptions.

Let us start by introducing some notation.

\subsection{Some notations} \label{sec:notation}

\begin{itemize}
\item For any measured space  $(S,\Sigma,\mu)$, the standard $\mathbb{L}^p(S)$--spaces of real-valued functions, with $p \in [1,\infty]$, are provided with their usual norm denoted by $\| \cdot \|_{\mathbb L^p(S)}$, or $\|\cdot \|_{\LL^p}$ whenever the space $S$ will be clear to the reader, and we  denote by $\langle f,g\rangle_{\LL^2}$ the inner product on $\mathbb{L}^2(S)$ between two functions $f$ and $g$. In more general cases, if the functions take values in some space $X$, the notation becomes $\mathbb{L}^p(S ;  X).$ Finally, 
the norm $\Vert \cdot\Vert_{\mathbb L^p(S) \to \mathbb L^p(S)}$ is the usual operator norm. 

\item The space of smooth real-valued functions with compact support in $\mathbb{R}^d$ is denoted by $C_0^\infty(\mathbb{R}^d)$. The space of functions of class $C^k$ with $k \in \mathbb{N}$ is denoted by $C^k(\mathbb{R}^d)$. The space of  \emph{bounded} functions of class $C^k$ is denoted by $C_b^k(\RR^d)$.
Finally, we denote by $\mathscr{C}_b^\eta$ the set of bounded real-valued functions defined on $\RR^d$ which are H\"older continuous with exponent $\eta \in (0,1]$. Its corresponding norm is denoted by 
\[[f]_\eta:= \|f\|_{\LL^\infty}+ \sup_{\substack{x,y\in\RR^d\\x \neq y}} \frac{|f(x)-f(y)|}{\|x-y\|^\eta}. \]
\item For any  $\alpha\in \mathbb{R}, p\geqslant 1$, we denote by
$\HH_{p}^{\alpha} $ the \textit{Bessel potential space}
\begin{equation}\label{defH}\HH_{p}^{\alpha}=\HH_{p}^{\alpha}(  \mathbb{R}^{d}):= \Big\{ u \in \mathcal{S}'(\mathbb R^d) \; ; \;  \mathcal{F}^{-1}\Big( \big(1+|\cdot|^{2}\big)^{\alpha/2%
}\; \mathcal{F} u(\cdot) \Big) \in \mathbb L^p(\mathbb R^d)\Big\},  \end{equation}
where $\mathcal Fu$ denotes the \textit{Fourier transform} of $u$. These spaces are endowed with their norm 
\[
\left\Vert u\right\Vert _{\alpha,p}%
:=\left\Vert \mathcal{F}^{-1}\Big((1+|\cdot|^{2})^{\alpha/2%
}\;\mathcal{F} u(\cdot)\Big)\right\Vert _{\mathbb L^{p}\left(  \mathbb{R}^{d}\right)  }%
<\infty.
\]
%Note that
% \[\left\Vert u\right\Vert _{0,2}=\left\Vert u\right\Vert_{\mathbb{L}^2(\mathbb{R}^d)} \]
%and moreover: \begin{itemize}
%\item  for any $\alpha\leqslant 0$, we have (using Plancherel's identity and the fact that $(1+|\cdot|)^{\varepsilon/2} \leqslant 1$)
%\[
%  \left\Vert u\right\Vert _{\alpha,2} 
%=\left\Vert    (1+|\cdot|^{2})^{\alpha/2}\;\mathcal{F} u(\cdot)
%\right\Vert_{\mathbb{L}^2(\mathbb{R}^d)} 
%\leq \left\Vert \mathcal{F} u
%\right\Vert_{\mathbb{L}^2(\mathbb{R}^d)} = \left\Vert u\right\Vert_{0,2} \; .
% \]
% \item for any $\alpha >0$, and any $p\geqslant 1$, we have 
% \[\left\Vert u\right\Vert _{\LL^p} \leq \left\Vert u\right\Vert _{\alpha,p}.\]
% 
% \end{itemize}
%
We recall some well-known Sobolev embeddings, see  \cite[p.~203]{Triebel},  which we will use in the forthcoming proofs. Let us fix $p,q \geqslant 1$. 
\begin{align}
&\text{If } \alpha > 0 & \text{then: } 
&\; \HH_p^\alpha \hookrightarrow \LL^p \quad \text{and} \quad \left\Vert \cdot \right\Vert _{\LL^p} \lesssim \left\Vert  \cdot \right\Vert _{\alpha,p}  \label{eq:SE0} \tag{SE0}  \\
%&\;  \; \left\Vert u\right\Vert _{\LL^p} \leq \left\Vert u\right\Vert _{\alpha,p} \; \;\; \quad\quad \text{for all } u \in \HH_p^\alpha. \label{eq:SE0} \tag{SE0} \\
&\text{If }  \lambda > \tfrac d p \; &\text{then: }
& \; \HH_{p}^{\lambda} \hookrightarrow \mathscr{C}_b^{\lambda-\frac d p} \quad \text{and} \quad  [ \cdot ]_{\lambda-d/p} \lesssim \| \cdot \|_{\lambda,p}  \label{eq:SE1} \tag{SE1} \\ 
%\quad\quad  \; \;\, \text{for all } u \in \HH_p^\lambda. \label{eq:SE1} \tag{SE1} \\ 
&\text{If } p> q & \text{then: }
&\; \HH_{q}^{\alpha+ \frac dq- \frac dp}\hookrightarrow  \HH_{p}^{\alpha} \quad \text{and} \quad
 \| \cdot \|_{\alpha,p}\lesssim \| \cdot \|_{\alpha + \frac{d}{q}-\frac{d}{p},q} .%\; \;\; \;\text{for all } u \in  \HH_{q}^{\alpha+ \frac dq- \frac dp}. 
 \label{eq:SE2} \tag{SE2}
\end{align}
Above we write $ f \lesssim g $ if there exists a universal constant $C>0$ such that $f \leqslant C g$.
\item For any $\alpha \in \RR$, we denote by $\dot{W}_{p}^{\alpha}$  the standard \emph{homogeneous Sobolev space} defined by
\[
\dot{W}_{p}^{\alpha}:= (-\Delta)^{-\alpha/2}\big(\LL^{p}(\mathbb{R}^d)\big).
\]
endowed with the norm $|f|_{\alpha,p}:= \| (-\Delta)^{\alpha/2} f\|_{\LL^p}$. We refer the reader to 
\cite{Gau, Triebel} for more properties of  this space.

We only highlight the fact that, for any $\eta\in \mathbb{R}$, the operator 
$(-\Delta)^{\eta/2}$ is an isomorphism from $\dot{W}_{p}^{\alpha}$ to $\dot{W}_{p}^{\alpha-\eta}$.

\item The expectation with respect to $\PP$ is denoted by $\EE$. We say that a function $f:\RR^d \to \RR$ is a probability density if
$f\ge 0 $  and  $\int_{\RR^d} f(x)dx=1$.

\end{itemize}

\subsection{Setting and hypothesis }
\label{sec:setting}

Let us now be more precise about several quantities which appear in the particle dynamics \eqref{itoassS0}, \textit{e.g.}~the interaction potential $V^N$, or the function $\mathbf{F}$. We will also specify the initial condition of the particle system, which is  random, and  supposed to be ``almost chaotic'' (see points \emph{1.}~and \emph{2.}~in Assumption \ref{assump} below). Finally, we will give the definition of the notion of solution to \eqref{eq:PDE}.

\medskip

Let us start with the function $\bF$. We first take  $M>0$ (which will be chosen ahead) and let $F_M: \RR \rightarrow\RR$ be a function in $C_b^2(\RR)$ such that
\[ F_M(x) = \begin{cases} M & \text{ if } x > M+1, \\
x & \text{ if } x \in [-M,M], \\
-M & \text{ if } x < -(M+1), \end{cases}\]
and moreover 
\[\|F_{M}^{\prime}\|_{\LL^\infty} \leq 1, \quad \text{ and } \quad \|F_{M}^{\prime\prime}\|_{\LL^\infty} <\infty, \quad \text{ therefore } \quad \|F_M\|_{\LL^\infty}\leq M+1.\]
We now build an $\RR^d$--valued function by defining
\begin{equation}\label{eq:defF0}
 \bF_{M}  : (x_{1},\dots, x_{d})\mapsto \big(F_{M}(x_{1}) ,\dots, F_{M}(x_{d})\big).
\end{equation}
We observe that  if $M=+\infty$,   then $\bF_{M}(x)=x$. The function $\bF$ driving the particle dynamics \eqref{itoassS0} will be equal to $\bF_M$ for some $M$ suitably chosen, in next Section \ref{sec:result}.

Let us now give the main assumptions on all the other parameters.

\begin{assumption}\label{assump}

Recall $\sigma \in (1,2)$. We  assume that 
\begin{enumerate}
\item \emph{(Interaction potential) } 

There exists $\beta >0$ and $V\in  C_{0}^{\infty}(\mathbb{R}^{d})$ a smooth probability density such that, for any $x \in \RR^{d}$, \[V^{N}(x)=N^{d\beta}V(N^{\beta}x)\; \emph{;}\]

\item \emph{(Initial condition) } 

There exist $p  > \max\big\{ \frac{d}{\sigma -1}\, ; \, 2\big\}$ and $\alpha>0$  with   \begin{equation}\label{eq:condalpha}\tfrac dp < \alpha < \sigma-1 <1 \end{equation} such that 
\begin{equation}
\sup_{N \in \NN} \EE\left[  \big\|  V^{N} \ast  \mu_{0}^{N} \big\|_{\alpha,p}^{2}   \right] \ 
<\infty \; ;  \label{initial cond}%
\end{equation}
\item \emph{(Linear operator)}

 The operator $K$ satisfies one of the following two properties:
\begin{itemize}
\item[\emph{\textbf{Case 1 --}}]
 There exists $\lambda$ such that $\frac{d}{p}<\lambda\leq \alpha$  and  \[K  :  \HH_{p}^{\alpha}\longrightarrow  \HH_{p}^{\lambda}\] is a linear  continuous  operator. 
In particular there exists $C_K>0$ such that, for any $u \in \HH_p^\alpha$, 
\begin{equation}  \|K(u)\|_{\lambda,p} \leqslant C_K \|u\|_{\alpha,p} \, . \label{eq:defCK}
\end{equation}
\item[\emph{\textbf{Case 2 --}}] There exists $\lambda$ such that $\frac{d}{p} < \lambda\leqslant \alpha $  and   
\begin{align*}
K :\; & \LL^{1}\cap \LL^{p} \longrightarrow  \HH_{p}^{\lambda}\\
:\;  & \dot{W}_{p}^{0}=\LL^{p}\longrightarrow  \dot{W}_{p}^{1} \quad 
\end{align*}    
is a linear  continuous  convolution  operator. In particular, there exists $\tilde C_K>0$ such that, for any $u \in  \LL^1\cap \LL^p$, 
\begin{align}
\|K(u)\|_{\lambda,p} &\leqslant \tilde C_K \big( \|u\|_{\LL^p} + \|u\|_{\LL^1}\big) \label{eq:defCK2}
\\ 
\text{and} \qquad |K(u)|_{1,p} &\leqslant \tilde C_K \|u\|_{\LL^p} \, .\label{eq:defCK3}
\end{align}
\end{itemize}

  \item \emph{(Relation between parameters)} 
  
  Finally, there exists $\delta > 1 - \frac\sigma 2$ such that the parameters  $(\beta,\alpha,p)$ satisfy: 
	\begin{equation} \label{eq:alpha} 
	0 < \beta <  \frac{1}{2(\alpha+d -\frac{d}{p}+ \delta)  }.
\end{equation}
Note that in the case $\sigma \to 2$, the parameter $\delta$ can be chosen arbitrarily small.

\end{enumerate}

\end{assumption}

\begin{remark} In \cite[Lemma 2.9]{flan3} the authors show 
that if $X_{0}^{i}  , i=1,...,N, $ are independent identically distributed r.v with common regular  probability density 
$u_{0}$ then  \eqref{initial cond} is  verified.  The  condition  \eqref{initial cond} 
 is necessary for the estimation of $u_{t}^{N}$ given in Lemma \ref{estimation2}.
\end{remark}

In Section \ref{sec:ex} we give several examples of convolution kernels which satisfy Assumption \ref{assump}.

\medskip

%\begin{assumption}\label{assump2}
%
%We assume that 
%\begin{enumerate}
%\item for any $x \in \RR^{d}$, $V^{N}(x)=N^{d\beta}V(N^{\beta}x)$ ;
%
%\item $V\in C_{0}^{\infty}(\mathbb{R}^{d})$ ;
%
%\item  $(1+  \alpha)/ \sigma < 1$, 
%
%\item there exists $p\geq (d\vee 2)$ and $\alpha\geq 0$ with $\alpha< 1-\sigma$ such that 
%
%
%\begin{equation}
%\sup_{N \in \NN} \EE\left[  \big\|  V^{N} \ast  \mu_{0}^{N} \big\|_{\alpha,p}^{2}   \right] \ 
%<\infty \; ;  \label{initial condbis}%
%\end{equation}
%
%\item  There exit $\frac{d}{p} < \lambda< 1$  such that   $K $ is a linear  continuous  convolution  operator, 
%$L^{1}\cap L^{p}  \longrightarrow  \HH_{p}^{\lambda}(\mathbb{R}^d)$ and $K :\dot{W}_{p}^{0}=L^{p}(\mathbb{R}^d) \longrightarrow  \dot{W}_{p}^{1} $. 
%
%
%
%  \item finally, the  parameters  $(\beta,\alpha,p)$ satisfy: 
%	\begin{equation} \label{eq:alphabis} 0 < \beta <  \frac{1}{2\alpha+2d -2\frac{d}{p}+ 2\delta  }.\end{equation}.
%	
%
%\end{enumerate}
%
%\end{assumption} 

From now on, we always work under Assumption \ref{assump}. We denote by  $(\mathcal{X},\|\cdot\|_{\mathcal{X}})$ the metric space 
\[ \big(\mathcal{X},\|\cdot\|_{\mathcal{X}}\big) := \begin{cases}
\big(\HH_{p}^{\alpha},\|\cdot\|_{\alpha,p}\big) & \text{ if we are in \textbf{Case 1}},\vphantom{\Big(}\\
\big(\LL^{1}\cap \LL^{p},\|\cdot\|_{\LL^1}+\|\cdot \|_{\LL^p}\big) & \text{ if we are in \textbf{Case 2}}. \vphantom{\Big(}\end{cases} \]
With these notations, point \emph{3.}~of Assumption \ref{assump} can be rewritten as follows: there exists $\lambda$ such that $\frac d p < \lambda \leqslant\alpha$, and a constant $\mathbf{C}_K >0$ such that, for any $u\in\mathcal{X}$,
\begin{equation}
\label{eq:dCK}
\|K(u)\|_{\lambda,p} \leqslant \mathbf{C}_K \|u\|_{\mathcal{X}}.
\end{equation}
and besides in \textbf{Case 2}
\begin{equation}
\label{eq:dCK2}
|K(u)|_{1,p} \leqslant \mathbf{C}_K \|u\|_{\LL^p}.
\end{equation}
\bigskip

We now define the notion of solutions to \eqref{eq:PDE}, which are understood in the mild sense:
\begin{definition}\label{def:defMild}
Given $u_0 \in \mathcal{X}$ and a time horizon $T>0$,
 a function $u$ defined on $[0,T] \times \RR^d$ is said to be a mild solution to \eqref{eq:PDE} on $[0,T]$ if
\begin{enumerate}
\item[(i)] $u\in C([0,T]; \mathcal{X}) $ and
\item[(ii)] $u$ satisfies the integral equation
\begin{equation}
\label{eq:mildKS}
u_{t} =  e^{t\cL} u_0 -  \int_0^t \nabla  e^{(t-s)\cL }  (u_{s}\,  K(u_{s}))\ \dd s, \quad \text{for any } 0 \leq t \leq T.
\end{equation}
\end{enumerate}
A function $u$ defined  on $\RR_+\times \RR^d$ is said to be a global mild solution to \eqref{eq:PDE} if it is a mild solution to \eqref{eq:PDE} on $[0,T]$ for any $T>0$.
\end{definition}

\begin{remark} \label{rem1} Let $u$ be a mild solution to  (\ref{eq:PDE}) on $[0,T]$, and assume that $K$ satifies \eqref{eq:dCK} for some parameters $\frac dp <\lambda \leq \alpha  $. Then,  using the Sobolev embedding \emph{(SE1)}   we have
\[
\| K(u)   \|_{\LL^{\infty}([0,T] \times \RR^{d})}\leq   \sup_{t\in [0,T]} \| K(u_t) \|_{\lambda, p}\leq  \bC_{K}  \sup_{t\in [0,T]} \|u_t \|_{\mathcal{X}}\, .
\]
Therefore,  $u$ is   also a mild solution of a companion PDE, in the sense that it satisfies the following integral equation
\begin{equation}
\label{eq:mildKSbis }
u_{t} =  e^{t\cL} u_0 -  \int_0^t \nabla  e^{(t-s)\cL }  \big(u_{s}\,  \bF_M(K(u_{s})) \big)\ \dd s, \quad\text{for any } 0 \leq t \leq T,
\end{equation}
where $\bF_M$  is the function defined in   (\ref{eq:defF0}) with $M$ being any fixed constant such that  \begin{equation} M \geqslant \mathbf{C}_{K} \sup_{t\in [0,T]} \|  u_t \|_{\mathcal{X}}. \label{eq:condM}\end{equation} 

\end{remark}

 To sum up, we have the following proposition:
\begin{proposition}\label{propu}
Given $u_0\in\cX$ and $T>0$, there exists a unique mild solution $u$ to the initial-valued PDE problem \eqref{eq:PDE} on $[0,T]$. 

Moreover, if $v \in C([0,T];\cX)$ satisfies the integal equation \eqref{eq:mildKSbis }, with the constant $M$ verifying $M \geqslant \mathbf{C}_K \sup_{t\in [0,T]} \|  u_t \|_{\mathcal{X}}$, then $v$ coincides with $u$.
\end{proposition}

\proof{We omit the proof, which follows from a standard contraction principle, see for instance 
 \cite[Theorem 13.2]{Lema} and \cite[Proposition 1.2]{Toma}.}

\subsection{Statement of the results} 
\label{sec:result}

From now on we fix $T>0$ and an initial data $u_0 \in \cX$. From Proposition \ref{propu}, we consider the unique mild solution $u \in C([0,T];\cX)$ to \eqref{eq:PDE} as given in Definition \ref{def:defMild}, and we settle
\begin{equation}\label{eq:defMT} M_T:= M_T(u)= \bC_K \sup_{t\in[0,T]} \|u\|_\cX.\end{equation}
Finally, we consider the particle dynamics $\{X_t^{i,N}\}_{i=1,\dots,N}$ ruled by \eqref{itoassS0} with $\bF:=\bF_{M}$ for some $M \geqslant M_T$.

\medskip

The main theorem of this paper gives an explicit bound on the distance between the mollified measure $u_t^N=V^N \ast \mu_t^N$, where $\mu_t^N$ is defined in \eqref{eq:mutN},  and the solution $u$. This distance is defined in two different ways which correspond to the two different cases. In \textbf{Case 1}, 
we will estimate the distance with respect to the norm $\|\cdot\|_{\alpha,p}$ in the space $\HH_{p}^\alpha$, but in \textbf{Case 2}, we need to ``distort" the norm. Let us then define: 
\begin{itemize} 
\item in \textbf{Case 1}, $\tilde{\mathcal{X}}:=\HH_{p}^\alpha$, endowed with the norm $\| \cdot \|_{\tilde{\mathcal{X}}}:=\|\cdot\|_{\alpha,p}$ 
\item in \textbf{Case 2}, $\tilde{\mathcal{X}} :=\HH_p^\alpha\cap \LL^1$, endowed with  $\| u \|_{\tilde{\mathcal{X}}}:= \|u\|_{\alpha,p} + \|K(u)\|_{\LL^p}$.
\end{itemize}
Note that in the second case $\|\cdot\|_{\tilde{\mathcal{X}}}$ is well defined from  \eqref{eq:SE0}.
\begin{theorem}\label{th:rate2sing}  
Under the above assumptions, there exist constants $c,C>0$ such that for any $N\in\mathbb{N}$,
\begin{align} \label{eq:concl}
   \sup_{t\in [0,T]}\EE\Big[  \| u_t^N-u \|_{\tilde{\mathcal{X}}}\Big]
&\leq  C \EE \Big[\|u^N_0- u_0 \|_{\tilde{\mathcal{X}}}\Big]   + c N^{-\varrho} ,
\end{align}
where
\begin{equation}\label{eq:rho}
\varrho = \min \left\{ \beta\Big(\lambda-\frac d p\Big)    \; ; \; \frac12\Big(1- 2 \beta\Big( d + \alpha -\frac {d} p + \delta \Big)\Big)    \right\} ,
\end{equation}
and  all the parameters have been introduced in Assumption \ref{assump}.
\end{theorem}

 The following corollary can be obtained   immediately. 

\begin{corollary}\label{cor}
If $u_0^N$ converges to $u_0$ in $\LL^1(\Omega \; ; \; \tilde{\mathcal{X}})$ as $N\to\infty$, then $u^N$ converges to $u$ in $\LL^\infty([0,T]\; ; \; \LL^1(\Omega \; ; \; \tilde{\mathcal{X}}))$ as $N\to\infty$.
\end{corollary}

\begin{remark} We note that we do not assume that $u\ge 0$. However, from Corollary \ref{cor}, if $u_0^N$ converges to $u_0$ in  $\LL^1(\Omega \; ; \; \tilde{\mathcal{X}})$, then the convergence of
$u^N$ towards $u$ in the space $\LL^\infty([0,T]\; ; \; \LL^1(\Omega \; ; \; \tilde{\mathcal{X}}))$ implies that $u\ge 0$. \end{remark}

By some interpolation arguments, and under an additional assumption, we can obtain a convergence result for the original empirical measure $\mu_t^N$: 
\begin{corollary}\label{Colo}  
Let $q\geqslant 1$ be the conjugate of $p$ (\emph{i.e.~}$\frac1p+\frac1q=1$), and let us define, for some $\epsilon>0$ (to be chosen ahead): 
\begin{equation}\label{eq:defr}\vartheta_{\epsilon}:=\frac{(p-q)(q-\epsilon)}{ q(p-q+\epsilon)}<1\qquad \text{and} \qquad r_{\epsilon}:=\alpha(2\theta_{\epsilon} -1)<\alpha.\end{equation}We denote by $p_\varepsilon$ the conjugate of $q-\varepsilon$ (\emph{i.e.}~$\frac{1}{p_\varepsilon}+\frac{1}{q-\varepsilon}=1$) and we take $\epsilon$ small enough such that \[r_{\epsilon}>\frac{d}{p}> \frac{d}{p_\varepsilon} \quad \text{ and } \quad \epsilon < q-\frac{dq}{q\sigma-d}.\] 
We assume, in addition to Assumption \ref{assump}, that  $ K( e^{t\mathcal{L}}  u_{0} ) \in \HH_{p_{\epsilon}}^{r_{\epsilon}} $ for any $t\geq 0$,  and furthermore that \[u_0\in C([0,T] ; \HH_{q-\epsilon}^{-r_{\epsilon}}).\] 
Then, the unique solution $u$ to \eqref{eq:PDE} belongs to
 \[{C}([0,T] ; \cX)\cap C([0,T] ;  \HH_{q-\epsilon}^{-r_{\epsilon}}).\] 
 Consider, as before, the dynamics of the particle system ruled by (\ref{itoassS0}) with $\bF=\mathbf{F}_M$, for some $M \geqslant M_T(u)$.
Then,  there exists constants $C,c>0$ such that for any $N\in\mathbb{N}$,
\begin{align*}
   \sup_{t\in [0,T]}\EE \Big[\| \mu_t^N-u \|_{-\alpha,q}\Big]
&\leq C \EE \Big[\|u^N_0- u_0 \|_{\tilde{\mathcal{X}}}^{1 -\vartheta_{\epsilon}}  \Big] 
+ c N^{-\hat{\varrho} } ,
\end{align*}
where
\begin{align*}
\hat{\varrho} = \min \left\{  \beta\Big(\alpha-\frac d p\Big)    \; ;\; \varrho (1- \vartheta_{\epsilon})      \right\}, \qquad \varrho \text{ being given by \eqref{eq:rho}}.
\end{align*}
\end{corollary}
%
%
%We denote the space $\tilde{X}:= \HH_{p}^{\alpha}( \mathbb{R}^{d}) \cap  L^{1} \mathbb{R}^{d})$ with the norm 
%\[ \|  u\|_{\tilde{X}} := \|  u \|_{p,\alpha} + \|  K(u) \|_{p} . \]
%
%\begin{theorem}\label{th:rate3sing}  
%Assume the hypothesis  \ref{assump2}.  
% Let $T$ be the maximal existence time for \eqref{eq:PDE} in the space $\mathcal{C}([0,T], (L^{1}\cap L^{\infty})(\RR^d))$ and fix $T\in(0,T_{\max})$. 
%In addition, let the dynamics of the particle system be given by (\ref{itoassS0}) with $M$ greater than  $M_{T}:=C_{K}\sup_{[0,T]} \|  u \|_{L^{1}\cap L^{p}}  $.
%Then  there exists a constant $C>0$ such that for all $N\in\mathbb{N}$,
%\begin{align*}
%   \sup_{t\in [0,T]}\EE  \| u^N-u \|_{\tilde{X}}
%&\leq  C_{T} \EE \|u^N_0- u_0 \|_{\tilde{X}}   + C N^{-\varrho} ,
%\end{align*}
%where
%\begin{align*}
%\varrho = \min \left(  \beta(\lambda-d/p)    , (1-  \beta( 2\alpha +2d -2d/p + 2\delta ))/2    \right)  .
%\end{align*}
%
%
%
%where $\delta>0$ and $1-\delta< \sigma/2$. 
%
%
%\end{theorem}

%\subsection{Related works}

%\label{sec:related}

\subsection{Examples of applications} \label{sec:ex}

Let us now give some concrete illustrations which satisfy Assumption \ref{assump}. First of all, we fix  \[\cL:=-(-\Delta)^{\sigma/2} \quad  \text{ with } \sigma\in(1,2).\]
The first three examples belong to \textbf{Case 1}, while the last two  satisfy \textbf{Case 2}.
\begin{enumerate}
\item \textsc{$2d$ SQG equation}. Let $d=2$. 
If one takes
\[K(u):=R^{\bot} u\] where $R=(-R_2, R_{1})$ are the usual Riesz transforms (\emph{i.e.}, the Fourier
multipliers  $-\xi_{1}/|\xi|$, $-\xi_{2}/|\xi|$), then  \eqref{eq:PDE} corresponds to the
well-known \emph{quasi-geostrophic SQG equation}, which arises from the geostrophic
study of strongly rotating flows, and  has been introduced to the mathematical community by Constantin, Majda, and Tabak \cite{Const},
 who noted structural similarity with the $3d$ incompressible Euler equation in the form of vortex stretching. 
Many results have been shown since then; we briefly
mention a few  of them. Local existence and uniqueness of 
classical solutions  in H\"older and Sobolev
spaces have been obtained by \cite{Chae,Const}. Global existence  
 of weak solutions is known in the spaces $\LL^{p}(\mathbb{R}^{2})$, for $p\in (\frac{4}{3}, \infty)$, see \cite{March, Res}.  For the subcritical case, \textit{i.e.}~$\sigma> 1$ which is the case under investigation here, and also the critical case $\sigma=1$, classical solutions are known to be globally
smooth  \cite{Const2,Kise}, and in fact, weak solutions are smooth for positive times \cite{Cafa,Const3}.
 For even more results, also in the  supercritical case, we refer to
 \cite{Cord2, Cord, Ferre2,  March, Kise, Wu,Wu2}.

  We recall that   the singular integral theory of Calderon and Zygmund gives the following property: for any $p\in (1,\infty)$, there
	is a constant $C_p>0$, such that, for any $f \in \HH_p^\alpha$
		\[
	\| R^{\perp} f\|_{\alpha,p}\leq C_p\| f\|_{\alpha,p}. 
	\]
	therefore  $K$ verifies \textbf{Case 1} with $\alpha=\lambda >\frac{2}{p}$.

\item \textsc{Turbulence model.} Let $d=2$  and now take \[K(u):=(-\Delta)^{\eta/2} R^{\bot} u
\] 
 for some general $\eta$. The case $\eta=0$ reduces to the previous SQG equation. In the general case, the scalar equation 
 \eqref{eq:PDE} is also called   \emph{$\eta$-turbulence model} 
 \cite{Iwa}. Following  \cite{Miaoa2}, the cases  $\sigma> \eta +1$,
 $\sigma=\eta +1$ and $\sigma < \eta +1$ correspond respectively to the subcritical, critical and supercritical cases.  This model has been introduced in \cite{Const4}, where Constantin et al.~considered first the cases  $\eta\in (-1,0)$ and $\sigma=\eta+1$, and obtained global regularity of  weak solutions.  Recently, some interesting  generalization of this equation has been exposed in \cite{Ferre}. 
 Further developments on regularity issues have covered extended cases, as for instance in  \cite{Miaoa, Kise}. 
 %the authors treat the whole critical case 
%$\eta\in (-1,1)$ and $\sigma= \eta+1$   thanks to the approach developed in \cite{Kise}: whenever $\eta\in (-1,0)$  global well-posedness of the smooth solution was proved; while when $\eta \in (0,1)$, it is obtained under the condition of small $\LL^{\infty}$ initial data. For the super-critical range $\eta \in (0,1)$ and $\sigma \in (0,\eta+1) $, Kiselev in \cite{Kise} also shows the eventual regularity of the weak solution. 
%On regularity issues, we refer to \cite{Miaoa2}  and for interesting  generalization of this equation see \cite{Ferre}. 
Here, similarly as in the previous example, our assumption covers the subcritical case. In that regards, when $\eta\in(0,1)$ and $\sigma \in (\eta+1,2]$ the authors in  \cite{Miaoa2} obtain the global well-posedness of the system with smooth initial data. 

 Here we observe that, for any $\eta>0$, and any $f\in\HH_p^\alpha$,
\[
\| (-\Delta)^{\eta/2} R^{\bot} f\|_{\alpha -\eta , p}\leq  C_\eta \|  R^{\bot} f\|_{\alpha, p} \leq C_\eta C_{p} \|  f\|_{\alpha , p}
\]
therefore  $K$ satisfies \textbf{Case 1} whenever $\eta$ is chosen such that $ \eta \in (0,\alpha-\frac 2 p ) \subset (0,1)$, and taking $\lambda = \alpha-\eta$.

\item \textsc{Burgers equation.}  When $d=1$, the case  $K(u):= u$  corresponds to the more
well-known  \emph{scalar  fractal Burgers equation}.  Equations of this type appear in the study of growing interfaces in the presence of self-similar hopping surface diffusion, see \cite{Mann}.  The Burgers equation with pure fractal diffusion  has been studied by numerous  authors, see for instance (without being exhaustive)
\cite{Biler3F,Biler4,Biler5, Kise2}. 

It is straightforward to check that this model satisfies \textbf{Case 1} with $\alpha=\lambda>\frac 1 p$. 

\item  \textsc{Keller-Segel models.} Let $d\geqslant 1$ be any dimension. Assume now that $K$ reads as \[K(u):=  k\ast u \qquad \text{where} \qquad k(x):= - c_{d} \frac{x}{|x|^{d}}, \quad x \in \RR^d,\  c_{d}>0.\]
This case  corresponds to the   \emph{fractional parabolic-elliptic  Keller-Segel models}, which have  been studied for instance in \cite{Biler2,Biler3,Escudero,Sule}. This system can describe directed movement of cells in response to the gradient of a chemical, but also several physical processes involving diffusion and interaction of particles.

We observe that $\nabla k$ is the Calderon-Zygmund operator and therefore
we have, for any $p \in (1,+\infty)$, \[ |k\ast u |_{1,p}= \| \nabla (k\ast u)\|_{\LL^p}\leq  C_p \|  u\|_{\LL^p}\] and, for any $p>d$, 
\begin{align*}%\|(k\ast u)\|_{\LL^\infty}&\leq C (\|u\|_{\LL^1} + \|u\|_{\LL^p}) \\ 
\| k\ast u\|_{1,p}\leq C(\| u \|_{\LL^1} +\| u \|_{\LL^p} ).\end{align*} Therefore,  since $\frac dp<\alpha<1$ we can write from the above the two conditions
\begin{align*}
 \| k\ast u\|_{\alpha,p} &\leq  \| k\ast u\|_{1,p} \leq  C(\| u \|_{\LL^1} +\| u \|_{\LL^p} )
\\ 
| k\ast u|_{1,p}&\leq C\|u\|_{\LL^p}. \vphantom{\Big(}
\end{align*}
Hence, this model satisfies \textbf{Case 2} with $\lambda=\alpha$.

\item \textsc{$2d$ Navier-Stokes equation.} Let $d=2$, and let us recall the following  \emph{generalized Navier-Stokes equation} which  describes the evolution of the velocity field
$u:\RR_+ \times \RR^2 \to \RR^2$ of an incompressible fluid with kinematic
viscosity coefficient $\kappa >0$: for any $(t,x) \in \RR_+ \times \RR^2 $,
\begin{equation}\label{Navier}
 \left \{
\begin{aligned}
    \partial_t u(t, x) &=-\kappa  (-\Delta)^{\sigma/2} u(t, x) - \big(u(t, x) \sca \nabla\big) u(t, x)  - \nabla p(t, x)
    \\\text{ div } u(t, x)&=0  \vphantom{ u^{\rm ini}}
\end{aligned}
\right .
\end{equation} with some initial data $u_0$ and 
where $\sca$ denotes the standard euclidean product in $\RR^d$.
It turns out that the associated (scalar) \textit{vorticity field} $ \xi= \partial_{1}u_{2} - \partial_{2}u_{1}:  \RR^2 \to \RR $ 
 satisfies a simple equation written as
\begin{equation}\label{Vorty}
\partial_{t} \xi  + u \sca \nabla \xi=-\kappa (-\Delta)^{\sigma/2} \xi, \qquad (t,x)\in \RR_+\times \mathbb{R}^{2}.
\end{equation}
The velocity field $u(t, x)$ can be reconstructed from the vorticity distribution $\xi(t, x)$ by the
convolution with the \textit{Biot-Savart kernel} $k$ as:
\begin{equation}
u(t,x)= \big(k \ast \xi(t,\cdot)\big)(x)= \frac{1}{2\pi} \int_{\mathbb{R}^{2}} \frac{(x-y)^{\bot}}{|x-y|^{2}} \ \xi(t,y) \dd y 
\end{equation}
where $(x_{1},x_{2})^{\bot}:=(-x_{2}, x_{1})$. Therefore, $\xi$ is solution to an equation of type \eqref{eq:PDE} with $K(\xi):=k\ast \xi$. This model  arises
naturally in fluid mechanics, see for instance  \cite{WuNS,Chae2,Ferre3,Lin}.
It is known that, when $p>2$, and $ \frac 2 p < \alpha < 1$,  
\begin{align*}
 \| k\ast \xi\|_{\alpha,p}& \leq  \| k\ast \xi\|_{1,p} \leq  C(\| \xi \|_{\LL^1} +\| \xi \|_{\LL^p} )
\\
| k\ast \xi|_{1,p}&\leq C \| \xi\|_{\LL^p}.
\end{align*} Therefore, this model satisfies \textbf{Case 2} with $\lambda=\alpha$.

\end{enumerate}

\section{Preliminaries}
\label{sec:prelim}

In this preparatory section, we present some classical definitions and related results on the symmetric stable Lévy processes, and we also provide some important estimates and identities which
will be used several times later on.

\subsection{Symmetric stable L\'{e}vy processes} \label{ssec:levy}

 We refer to  \cite{App, Kunita, Sato} for all the material presented here.

\begin{definition}[\textbf{L\'{e}vy process}] \label{appendix: def levy process} \textit{A process $L=(L_t)_{t \geq 0}$ with values in $\RR^d$ defined on a probability space $(\Omega, \fF, \PP)$ is a \emph{L\'{e}vy process} if the following conditions are fulfilled:
\begin{itemize}
\item[1.] $L$ starts at 0, $\PP-$a.s.: $\PP(L_0=0)=1;$ 
\item[2.] $L$ has independent increments: for $k \in \NN$ and $0 \leq t_0 < \dots < t_k$, 
\begin{align*}
L_{t_1} - L_{t_0},  \dots, L_{t_k} - L_{t_{k-1}} \quad \text{are independent };
\end{align*}
\item[3.] $L$ has stationary increments: if $ s \leq t$, then $L_{t} - L_s $ has the same distribution as $ L_{t-s}\; ;$
\item[4.] $L$ is stochastically continuous: for all $t \geq 0$ and $\varepsilon>0$, $
\displaystyle \lim_{s \rightarrow t} \PP (|L_t - L_s| > \varepsilon) =0.
$
\end{itemize}}
\end{definition}

%We recall
%\begin{proposition}{\emph{\cite[Theorem 2.1.8]{App}}} Every L\'{e}vy process has a c\`{a}d-l\`{a}g modification that is itself a L\'{e}vy process.
%\end{proposition}

%Following \cite[Section 2.4]{App} and  \cite[Chapter 4]{Sato} we state the \textit{L\'{e}vy-It\^{o} decomposition theorem} which characterizes  the paths of a  L\'{e}vy process in the following way.

Furthermore, we assume in this paper that  $(L_t)_{t\geq 0}$ is a \emph{symmetric $\sigma$-stable} Lévy process for some $\sigma \in (1,2)$. These processes can be completely defined via their characteristic function, which is given by (see \cite{Sato} for instance)
\[
\EE\big[ e^{i  \xi \sca L_t}\big]= e^{-t\psi(\xi)}, \qquad \xi\in\RR^d, t\geq 0
\]
where 
\begin{equation}\label{eq:psi}
\psi(\xi):= \int_{\RR^{d}} \big(1-e^{i (\xi \sca z)}+ i (\xi \sca z) {\bf 1}_{\{|z|\leq1\}}\big) \; \dd\nu(z), \qquad \xi \in \RR^d
\end{equation}
and the L\'{e}vy measure $\nu$  reads as
\[
\nu(U):=\int_{\mathbb{S}^{d-1}} \int_{0}^{\infty} \frac{{\bf 1}_{U}(r\theta)}{r^{\sigma+d}} \; \dd r \; \dd \mu(\theta) , \qquad U \in \mathcal{B}(\RR^{d}-\{0\}),
\]
where $\mathcal{B}(X)$ denotes the Borel sets of $X$, and $\mu$ is some symmetric finite measure concentrated on the unit sphere $\mathbb{S}^{d-1}$, called \textit{spectral measure}
of the stable process $(L_{t})_{t\geq 0}$.  
We also assume the following additional property: for some  $C_{\sigma}> 0$, 
\begin{equation} 
\psi(\xi)\geq C_{\sigma} |\xi|^{\sigma}, \qquad \text{ for any }\; \xi\in \RR^{d}.
\end{equation}
We remark that the above condition is equivalent to the fact that the support of the spectral measure $\mu$ is not contained in the 
proper linear subspace of $\RR^{d}$, see \cite{Priola}.

The \textit{L\'{e}vy-It\^{o} decomposition Theorem} (\cite[Section 2.4]{App} and  \cite[Chapter 4]{Sato}) implies that $L_t$ can the be written as
\begin{equation}\label{levy}
L_t= \int_0^t \int_{|z|>1} z\; \mathcal{N}(\dd s\dd z) + \int_0^t  \int_{0 < |z| < 1}  z\; \tilde{\mathcal{N}}(\dd s\dd z),
\end{equation}
where the  Poisson  random measure $\mathcal{N}$ is defined by
\begin{equation}\label{eq:poisson}
\mathcal{N}((0,t]\times U)=\sum_{s\in (0,t]} {\bf 1}_{U} (L_s-L_{s_{-}}), \qquad  U \in \mathcal{B}(\RR^{d}-\{0\}), \; t>0,
\end{equation}
and finally the compensated Poisson  random measure is given by  \begin{equation}\label{eq:tildeN}\tilde{\mathcal N}((0,t]\times U)= \mathcal{N}((0,t]\times U)- t \nu(U).\end{equation} 
We recall the following well-known properties about the symmetric $\sigma$-stable Lévy processes (see for instance \cite[Proposition 2.5]{Sato} and 
\cite[Section 3]{Priola}).

%\newpage

\begin{proposition} Let $t>0$ be fixed and let $\mu_t$ be the probability law of the symmetric $\sigma$-stable process $L_t$. The following properties hold:
\begin{enumerate}
\item  \emph{(Scaling property)}.    For any $\lambda>0$,  $L_t$ and $\lambda^{-\frac{1}{\sigma}}L_{\lambda t}$ have the same finite dimensional law. 
 In particular, for any $t>0$ and $A\in \mathcal{B}(\RR^{d})$,  $\mu_t(A)=\mu_1(t^{-\frac{1}{\sigma}}A)$.

\item \emph{(Existence of smooth density)}.  For any $t>0$, $\mu_t$ has a smooth density $\rho_t$ with respect to the Lebesgue measure,  which is given by
\[
\rho_t(x)=\frac{1}{(2\pi)^{d}}  \int e^{-i ( x \sca \xi)} \ e^{-t\psi(\xi)}   \  \dd\xi, \qquad x \in \RR^d.
\]
Moreover $\rho_{t}(x)=\rho_{t}(-x)$ and for any $k\in \mathbb{N}$,  $\nabla^k \rho_t \in \LL^{1}(\RR^{d})$. 
\item  \emph{(Moments)}. For any $t> 0$,  $\mathbb{E}\big[|L_t|^{\beta}\big]< \infty$ if and only if $\beta < \sigma$.
\end{enumerate}
Besides, the infinitesimal generator of the process $(L_t)_{t\geqslant 0}$ is given by $\mathcal{L}$ defined as follows: for any $\phi \in C_0^\infty(\RR^d)$, 
\begin{equation} \mathcal{L}\phi(x) = \int_{\RR^{d}-\left\{  0 \right\}}  \big(\phi(x+ z)-\phi(x) -{\bf 1}_{\{|z|\leq 1\}} \nabla\phi(x)\sca z\big) \;  \dd\nu(z). \label{OpeL}\end{equation}

\end{proposition}

For instance, the  fractional Laplacian $-(-\Delta)^{\frac{\sigma}{2}}$ satisfies \eqref{OpeL} with 
\[ \dd\nu(z) = \frac{\kappa_\sigma \; \dd z}{|z|^{1+\sigma}}, \qquad \text{for some } \kappa_\sigma>0.\]

\subsection{Maximal function and semi-group}
\label{sec:max} 

We will use in the proofs two analytic tools which we expose in this paragraph: the first one is the so-called \emph{maximal function}, and the second one is an estimate of the operator norm of some semi-groups. 

First of all, let $f$ be a locally integrable function on $\RR^{d}$.  The \emph{Hardy-Littlewood maximal function} $\mathbb{M}f$ is
defined by
\[
\mathbb{M}f(x):=\sup_{0< r< \infty} \bigg\{\frac{1}{|\mathscr{B}_r|} \int_{\mathscr{B}_r} f(x+y) \ \dd y\bigg\}, \qquad x \in\RR^d,
\]
where $\mathscr{B}_r=\left\{ x\in\RR^{d} \ : |x|< r   \right\}$. The following results can be found in \cite{Stein}.
\begin{lemma} \label{maxi}For any $f\in \HH_{1}^1$, there exists a constant $C_d>0$ and a Lebesgue-zero set $E\subset \RR^d$ such that, for any  $x,y \in \RR^d - E$,
\[
|f(x)-f(y)|\leq C_d \; |x-y|\; \Big(\mathbb{M}|\nabla f|(x)   +  \mathbb{M}|\nabla f|(y)\Big).\]
Moreover, for all $p>1$   there exists a constant  $C_{d,p}>0$ such that for any  $f\in \LL^{p}(\RR^{d})$
\[
 \|\mathbb{M} f\|_{\LL^{p}} \leq C_{d,p}  \;   \| f\|_{\LL^{p}}.
\]
\end{lemma}

Next, recall the definition of $\mathcal{L}$ given by  \eqref{OpeL} The family of operators $\{e^{t\mathcal{L}}\}_{t\geqslant 0}$ 
defines a Markov semi-group on each space $ \HH_{p}^{\alpha}$ ; with little abuse of notation, we write $e^{t\mathcal{L}}$
for any value of $\alpha$.

Let us consider  the operator $A:D(  A)  \subset
\LL^{p}\rightarrow \LL^{p}  $ defined as $Af=\Delta f$. Then, for any $\kappa\geq 0 $,  the  fractional powers $(
I-A)^{\kappa/2}$ are well defined  and
$\Vert (  I-A)^{\kappa/2}f\Vert _{\LL^p}$ is equivalent to the norm $\|f\|_{\kappa,p}$ of $\HH_{p}^{\kappa}  $. We have the following:

\begin{proposition} \label{prop} Assume $p\in(1,+\infty)$. For any
$\kappa\geq 0$, $\sigma \in (1,2)$, and given $T>0$, there is a constant $C_{\kappa,\sigma,T,p}>0$ such that, 
for any $t\in(0,T]$, 
\begin{equation}\label{s1}
\left\Vert \left(  I-A\right)^{\kappa}e^{t\mathcal{L}}\right\Vert _{\LL^{p}\rightarrow
\LL^{p}}\leq\frac{C_{\kappa,\sigma, T,p}}{t^{2 \kappa/\sigma}}.%
\end{equation}
\end{proposition}

\begin{proof}
   The proof is  the same as the one given in  
\cite[Proposition 7]{Simon-O}, which is written for $p=2$, but can be easily extended to the case of general $p \in (1,+\infty)$.
\end{proof}

 Finally we observe that  (by a standard convolution inequality): for any $f \in \HH_p^\kappa$ and $1\leq q< p\leq  \infty $,  we have 
\begin{equation}
\| e^{t\mathcal{L}}f \|_{\kappa, p}\leq  
 \frac{C_{\kappa,\sigma,T,p}}{t^{d/\sigma(1/q-1/p)}} \| f \|_{\kappa, q}. \label{eq:semigroup}
\end{equation}

\subsection{The equation for $u_{t}^{N}$ in
mild form} \label{sec:eq}

The starting point of the proof consists in using Itô's formula, which we apply to any observable of the form $\phi(  X_{t}^{i,N})$, for any  smooth and compactly supported test function $\phi\in C_{0}^{\infty}(\mathbb{R}^{d})$. We use the standard notation \[\langle \mu_t^N,\phi \rangle := \frac{1}{N}\sum_{i=1}^N \phi(X_t^{i,N}). \]
One can easily check that the empirical measure $\mu_{t}^{N}$ satisfies\footnote{Recall that the operator $K$ is linear.}%
\begin{align}
\left\langle \mu_{t}^{N},\phi\right\rangle = & \left\langle \mu_{0}^{N}%
,\phi\right\rangle +\int_{0}^{t}\Big\langle \mu_{s}^{N},\bF\big(K(V^{N}\ast
\mu_{s}^{N})\big)\;\nabla\phi\Big\rangle\; \dd s \notag \\
&+ \frac{1}{2}\int_{0}^{t}\left\langle \mu_{s}^{N}, \mathcal{L} \phi\right\rangle
\dd s \notag \\ & +   \frac{1}{N}\sum_{i=1}^{N}\int_{0}^{t} \int_{\RR^{d}-\{0\}} \big(\phi(X_{s_{-}}^{i,N}+z) - \phi(X_{s_{-}}^{i,N})\big)\;   {\mathcal N}^{i}(\dd s\dd z)
\label{ident fro S^N},
\end{align}
where each $\mathcal{N}^i$ is the Poisson random measure associated with the process $(L_t^i)_{t\geqslant 0}$, see \eqref{eq:poisson}. Let us now write the equation satisfied by $u_t^N$ in \emph{mild form}. We fix $x\in\RR^d$. 
 From \eqref{ident fro S^N}, we obtain by taking the test
function $\phi_{x}\left(  y\right) : =V^{N}\left(  x-y\right)  $ the following identity: 
\begin{align*}
  u_{t}^{N}  (x)    =&\; u_{0}^{N}  (x)  +\int_{0}^{t}\left\langle \mu_{s}
^{N},\bF(K(u_{s}^{N}))\nabla( V^{N}
)  \left(  x-\cdot\right)  \right\rangle \dd s
\\ & 
  +\frac{1}{2}\int_{0}^{t}\mathcal{L}\big( u_{s}%
^{N}\big)  (x)\dd s
\\ & 
 + \frac{1}{N}\sum_{i=1}^{N} \int_{0}^{t} \int_{\RR^{d}-\{0\}} \Big\{  (
e^{h\mathcal{L}} V^{N} )  \big(  x-X_{s_{-}}^{i,N}+ z  \big)  -   
   (V^{N})\big( x-X_{s_{-}}^{i,N}\big)  \Big\} \tilde{\mathcal N}^{i}(\dd s\dd z),
\end{align*}
where $\tilde{\mathcal{N}}^i$ is the compensated Poisson measure defined in \eqref{eq:tildeN}. 
In order to simplify notations, let us write, in the sequel,%
\[
\left\langle \mu_{s}^{N},\bF(K( u_{s}^{N}))\nabla(V^{N} )  \left(  x-\cdot\right)  \right\rangle =:\Big(
\nabla V^{N}   \ast\big(
\bF(K( u_{s}^{N}))\mu_{s}^{N}\big)  \Big)  \left(  x\right).
\]
and likewise for similar expressions. Following a standard procedure, see for instance \cite{flan}, we may
rewrite this equation in mild form as follows:
\begin{align*}
u_{t}^{N}    = &\; e^{t\mathcal{L}}u_{0}^{N}  +\int_{0}^{t}e^{\left(  t-s\right) \mathcal{L}%
}\left(  \nabla  V^{N} \ast\left(
\bF(K( u_{s}^{N}))\mu_{s}^{N}\right)  \right)  \dd s\\
&  +    \frac{1}{N}\sum_{i=1}^{N} \int_{0}^{t}  e^{\left(  t-s\right)\mathcal{L}} \int_{\RR^{d}-\{0\}}  \Big\{ ( V^{N} )  \big(  \cdot -X_{s_{-}}^{i,N}+ z  \big) -   ( V^{N} )\big( \cdot-X_{s_{-}}^{i,N}\big) \Big\} \tilde{\mathcal N}^{i}(\dd s\dd z).
\end{align*}
When one writes the explicit convolution formula for $e^{\left(  t-s\right)
\mathcal{L}}$, we see that $e^{\left(  t-s\right)  \mathcal{L}}\nabla f=\nabla e^{\left(
t-s\right) \mathcal{L} }f$. Using the semi-group property, we may
also write
\begin{align}
u_{t}^{N}&=e^{t\mathcal{L}}u_{0}^{N}+\int_{0}^{t}\nabla
e^{\left(  t-s\right) \mathcal{L}}\left(  V^{N} \ast\left(
\bF(K( u_{t}^{N}))\mu_{s}^{N}\right)  \right)  \dd s + \mathcal{M}_t^N(\cdot)  \label{eq:mild2}\end{align}
%
%
%\frac{1}{N}\sum_{i=1}^{N} \int_{0}^{t} \int_{\RR^{d}-\{0\}}  \Big\{ \big(
%e^{(t-s)\mathcal{L}} V^{N} \big)  \big(  x-X_{s_{-}}^{i,N} + z \big) \notag %\\
%&  \qquad \qquad \qquad \qquad \qquad \qquad \qquad  -   
 %  \big(e^{(t-s)\mathcal{L}} V^{N} \big)\big( x-X_{s_{-}}^{i,N}\big) \Big\} \tilde{\mathcal N}^{i}(\dd s\dd z).
%\end{align}
where for any $x \in \RR^d$, $(\mathcal{M}_t^N(x))_{t\geq 0}$ is the martingale given by
  \begin{multline}
\label{eq:martingale}
\mathcal{M}_t^N(x):=  \frac{1}{N}\sum_{i=1}^{N} \int_{0}^{t} \int_{\RR^{d}-\{0\}}  \Big\{ \big(
e^{(t-s)\mathcal{L}} V^{N} \big)  \big(  x-X_{s_{-}}^{i,N} + z  \big) \\ -   
   \big(e^{(t-s)\mathcal{L}} V^{N} \big)\big( x-X_{s_{-}}^{i,N}\big) \Big\} \tilde{\mathcal N}^{i}(\dd s\dd z).
\end{multline}

\section{Proof of Theorem  \ref{th:rate2sing} }

\label{sec:proof1}
From now on we denote by $C>0$ a generic constant which may change line to line, and which may depend only on the following parameters: the time horizon $T>0$, the data $\mathbf{F}$, $K$, $\sigma$, $\beta$, $\alpha$, $p$.

The strategy of the proof consists in using both mild formulas (\ref{eq:mild2}) for $u^N$ and (\ref{eq:mildKSbis })  for $u$: after subtracting one to the other, one estimates each term by triangular inequality. More precisely, in view of     (\ref{eq:mild2})  and (\ref{eq:mildKSbis }) it comes: for any $x\in\RR^d$,
\begin{align}
u^N_{t}(x) - u_{t}(x) &= e^{t\mathcal{L}} (u^N_{0} - u_{0})(x)  \vphantom{\int_0^1} \notag\\
&\quad + \int_{0}^t \nabla  e^{(t-s)\mathcal{L}} 
\Big( \big\langle \mu_{s}^N,  V^N (x-\cdot) \bF( K( u^N_s))(.) \big\rangle - u_{s}  \bF(K(u_s))(x) \Big) \, \dd s \notag\\
&\quad +  \mathcal{M}_t^N(x) \vphantom{\int_0^1}. \notag \end{align}
Therefore
\begin{align} u^N_{t}(x) - u_{t}(x)
&= e^{t\mathcal{L}} (u^N_{0} - u_{0})(x) -  \int_0^t  \nabla  e^{(t-s)\mathcal{L} } \left(u_{s}  \bF(K( u_s)) - u^{N}_{s}  \bF(K(u^{N}_s))\right)(x) ~ \dd s  \notag\\
&\quad + E_{t}(x) + \mathcal{M}^N_{t}(x), \vphantom{\int_0^1} \label{eq:mild2bis}
\end{align}
where $\mathcal{M}_t^N(x)$ is defined in \eqref{eq:martingale} and  we have set
\begin{align}\label{eq:defM}
E_{t}(x)&:= \int_0^t \nabla e^{(t-s)\mathcal{L} } \Big\langle \mu_{s}^{N},  V^N (x-\cdot) \Big( \bF\big( K(u^N_s)\big)(\cdot)-\bF\big( K( u^N_s)\big)(x)\Big)\Big\rangle \ \dd s. 
\end{align}
Then we have, by triangular inequality,
\begin{align}
\left\Vert  u^N_{t}- u_{t} \right\Vert_{\alpha,p}
&\leq \left\Vert  e^{t\mathcal{L}} (u^N_{0} - u_{0}) \right\Vert_{\alpha,p} \vphantom{\int_0^1} \label{eq:previous0} \\ & \quad  +  
\int_0^t  \left\Vert  \nabla e^{(t-s)\mathcal{L} } \left(u_{s}  \bF(K( u_s)) - u^{N}_{s}  \bF(K(u^{N}_s))\right)\right\Vert_{\alpha,p} ~ \dd s \label{eq:previous1}\\
&\quad + \big\Vert   E_{t} \big\Vert_{\alpha,p} +\left\Vert  \mathcal{M}^N_{t}\right\Vert_{\alpha,p}. \vphantom{\int_0^1} \label{eq:previous}
\end{align}
We divide the rest of this  section into two paragraphs corresponding to the two \textbf{Cases} of Assumption \ref{assump}.

\paragraph*{\textbf{Case 1:}} We start with an estimate of the second term \eqref{eq:previous1}. 
We observe  that 
\begin{align*}
 \mathfrak{I}_t^N&:=\int_0^t  \Big\Vert  \nabla  e^{(t-s)\mathcal{L} } \left(u_{s}  \bF(K( u_s)) - u^{N}_{s}  \bF(K(u^{N}_s))\right)\Big\Vert_{\alpha,p} ~ \dd s \\ 
 &\leq C \int_0^t  \Big\Vert  (I-A)^{(1+\alpha)/2} \; e^{(t-s)\mathcal{L} } \Big\Vert_{\LL^p \to \LL^p} \; \big\|u_{s}  \bF(K( u_s)) - u^{N}_{s}  \bF(K(u^{N}_s))\big\Vert_{\LL^p} ~ \dd s.
\end{align*}
Therefore, using Proposition \ref{prop} and then triangular inequality we obtain
\begin{align*} \mathfrak{I}_t^N &  
\leq  C  \int_0^t   \frac{1}{(t-s)^{(1+\alpha)/\sigma}} 
\Big\Vert   u_{s}  \bF(K( u_s)) - u^{N}_{s}  \bF(K(u^{N}_s))\Big\Vert_{\LL^p} ~ \dd s 
\\ &
\leq  C  \int_0^t   \frac{1}{(t-s)^{(1+\alpha)/\sigma}} 
\left\Vert   u_{s} \Big(\bF(K( u_s)) - \bF(K(u^{N}_s))\Big)  \right\Vert_{\LL^p} ~ \dd s 
\\ & \qquad  \qquad
+   C  \int_0^t   \frac{1}{(t-s)^{(1+\alpha)/\sigma}} 
\left\Vert  (u_{s}-u_{s}^{N})  \bF(K(u^{N}_s))  \right\Vert_{\LL^p} ~ \dd s \end{align*}
and we bound \begin{align*}
\mathfrak{I}_t^N   & \overset{(\star)}{\leq}  C  \int_0^t   \frac{1}{(t-s)^{(1+\alpha)/\sigma}} \left\Vert u_{s} \right\Vert_{\LL^\infty}
\left\Vert K( u_s-u^{N}_s)  \right\Vert_{\LL^p} ~ \dd s 
\\ & 
\qquad \qquad +   C  \int_0^t   \frac{1}{(t-s)^{(1+\alpha)/\sigma}} 
\left\Vert  u_{s}-u_{s}^{N} \right\Vert_{\LL^p} ~ \dd s 
\\&
\overset{(\diamond)}{\leq}   C  \int_0^t   \frac{1}{(t-s)^{(1+\alpha)/\sigma}} 
\left\Vert  u_{s}-u_{s}^{N} \right\Vert_{\alpha, p } ~ \dd s 
\end{align*}
where: \begin{itemize} \item to obtain $\overset{(\star)}{\leq}$ we used that $\bF$  is bounded and Lipschitz continuous; \item to obtain $\overset{(\diamond)}{\leq}$ we used: for the first term, the fact that $u$ is bounded since it belongs to $C([0,T], \HH_{p}^\alpha)$ (Proposition \ref{propu}), together with the embedding \eqref{eq:SE0}; and for the second term, the embedding \eqref{eq:SE0} together with assumption \eqref{eq:dCK}, which yields
\[ \|K(u_s-u_s^N)\|_{\LL^p} \leqslant  \|K(u_s-u_s^N)\|_{\lambda,p} \leqslant \mathbf{C}_K\|u_s-u_s^N\|_{\alpha,p}.\] \end{itemize} 
Now, we estimate $\big\|   E_{t} \big\|_{\alpha,p}$ in the same way, as follows:
\begin{align*}
\big\Vert E_{t} \big\Vert_{\alpha, p }
&
\leq  C \int_0^t   \frac{1}{(t-s)^{(1+\alpha)/\sigma}} 
\left\Vert \Big\langle \mu_{s}^N,  V^N (x-\cdot) \Big( \bF\big( K(u^N_s)\big)(x)-\bF\big( K( u^N_s)\big)(\cdot)\Big)\Big\rangle 
\right\Vert_{\LL^p }  \dd s
\\ & 
\overset{(\star)}{\leq} C \int_0^t   \frac{1}{(t-s)^{(1+\alpha)/\sigma}} 
\left\Vert \Big\langle \mu_{s}^N,  V^N (x-\cdot) \; \big| K(u^N_s)(x)- K( u^N_s)(\cdot) \big| \Big\rangle 
\right\Vert_{ \LL^p }  \dd s
\\ & 
\overset{(\diamond)}{\leq}  \frac{C}{N^{\beta(\lambda-d/p) }} \int_0^t   \frac{1}{(t-s)^{(1+\alpha)/\sigma}} 
\left\Vert \langle \mu_{s}^N,  V^N (x-\cdot)  \rangle 
\right\Vert_{\LL^p }  \; \left\Vert u_s^N \right\Vert_{\alpha,p} \dd s
\\& 
= \frac{C}{N^{\beta(\lambda-d/p) }} \int_0^t   \frac{1}{(t-s)^{(1+\alpha)/\sigma}} 
\left\Vert  u_{s}^{N} \right\Vert_{\LL^p }   \ \left\Vert  u_{s}^{N} \right\Vert_{\alpha, p }  \dd s, 
\end{align*}
where: \begin{itemize}
\item to obtain $\overset{(\star)}{\leq}$ 
 we  used the fact that $\bF$  is   Lipschitz continuous;
 \item to obtain $\overset{(\diamond)}{\leq}$ we used: first, 
the Sobolev embedding \eqref{eq:SE1} with the parameter $\lambda > \frac d p$ coming from Assumption \ref{assump}, and then, the assumption \eqref{eq:dCK} on the operator $K$.
\end{itemize}
Therefore, from all the previous estimates we deduce that 
\begin{align}
\left\Vert  u^N_{t}- u_{t} \right\Vert_{\alpha,p}
&\leq \left\Vert  e^{t\mathcal{L}} (u^N_{0} - u_{0}) \right\Vert_{\alpha,p} +  
 C  \int_0^t   \frac{1}{(t-s)^{(1+\alpha)/\sigma}} 
\left\Vert  u_{s}-u_{s}^{N} \right\Vert_{\LL^p} ~ \dd s  \notag\\
&\quad +  \frac{C}{N^{\beta(\lambda-d/p) }} \int_0^t   \frac{1}{(t-s)^{(1+\alpha)/\sigma}} 
\left\Vert  u_{s}^{N} \right\Vert_{\LL^p }   \ \left\Vert  u_{s}^{N} \right\Vert_{\alpha, p }  \dd s  +\left\Vert  \mathcal{M}^N_{t} \right\Vert_{\alpha,p}.  \label{eq:last}
\end{align}
We take expectation on both sides: in the time integral appearing in \eqref{eq:last}, we first use the Cauchy-Schwarz inequality in order to bound 
\begin{equation}\label{eq:cs}\EE\big[\|u_s^N\|_{\LL^p} \| u_s^N\|_{\alpha,p} \big] \leqslant \EE\big[\|u_s^N\|_{\LL^p}^2\big]^{1/2} \; \EE\big[\|u_s^N\|_{\alpha,p}^2\big]^{1/2}. \end{equation}
We need an intermediate lemma: 
\begin{lemma}\label{estimation} Take any $\kappa$ such that $-1 \leq \kappa\leq \alpha < \sigma-1$, recalling that Assumption \ref{assump} is verified. There exists a constant $C>0$ such that, for any $t\in[0,T]$ and $N\in\mathbb{N}$, it holds
\[
\mathbb{E}\big[  \big\Vert \left(  I-A \right) ^{\kappa/2}u_{t}
^{N}\big\Vert^2_{\LL^{p}}\big]  \leq
C.
\]
\end{lemma}
The proof of the above lemma does not involve any difficulty, and we postpone it to Appendix \ref{sec:app}. 
We use Lemma \ref{estimation} twice: once with $\kappa=0$ and once with $\kappa=\alpha$, and we finally obtain from \eqref{eq:cs} that $\EE\big[\|u_s^N\|_{\LL^p} \|u_s^N\|_{\alpha,p} \big]$ is bounded by a constant with does not depend on $N$. 

It remains to estimate  $\EE[\|\mathcal{M}^N_t\|_{\alpha,p}]$. This is also done \emph{via} an intermediate lemma, stated as follows, and whose proof is  postponed to Appendix \ref{sec:app}, since it also uses standard arguments: 
\begin{lemma} \label{lemma martingale 2} Take any $\kappa$ such that $-1 \leq \kappa \leq \alpha < \sigma -1$, recalling that Assumption \ref{assump} is verified.  Then, there exists a constant $C>0$, such that,  for any  $N\in\mathbb N$ and $t\geq 0$ it holds,
\[
\EE \Big[\big\Vert  \left(  I-A\right)^{\kappa/2} \mathcal{M}_t^{N}    \big\Vert _{\LL^{p}}^{2} \Big]\leq
\frac{C}{N^{  1- 2 \beta( \kappa +d -d/p + \delta )     }} . 
\]
\end{lemma}
We apply the above with $\kappa=\alpha$ in order to bound $\EE[\|\mathcal{M}_t^N\|_{\alpha,p}]$ (recalling that $\EE[|X|] \leqslant \EE[X^2]^{1/2}$). To sum up, since $1+\alpha < \sigma$ from Assumption \ref{assump}, we obtain
\begin{align*}
\EE \big[ \left\Vert  u^N_{t} - u_{t} \right\Vert_{\alpha,p} \big]
&\leq C\;\EE \big[\left\Vert  u^N_{0} - u_{0} \right\Vert_{\alpha,p} \big]+  
 C  \int_0^t   \frac{1}{(t-s)^{(1+\alpha)/\sigma}} 
\EE \big[\left\Vert  u_{s}-u_{s}^{N} \right\Vert_{\LL^p}\big] ~ \dd s  \\
&\quad +  \frac{C}{N^{\beta(\lambda-d/p) }}  +  \frac{C}{N^{\frac12(1- 2 \beta( \alpha +d -d/p + \delta ))}} 
\end{align*}
and from Gronwall's Lemma we may now conclude 
\begin{align*}
\EE \big[\left\Vert  u^N_{t} - u_{t} \right\Vert_{\alpha,p}\big]
&\leq C \; \EE\big[ \left\Vert  u^N_{0} - u_{0} \right\Vert_{\alpha,p}\big] +  \frac{C}{N^{\beta(\lambda-d/p) }}  + \frac{C}{N^{\frac12(1-  2\beta( \alpha +d -d/p + \delta ))}} . 
\end{align*}
Recalling Assumption \ref{assump}, we obtain exactly \eqref{eq:concl}, and this concludes the proof of Theorem \ref{th:rate2sing} in \textbf{Case 1}.

\paragraph*{\textbf{Case 2:}} The proof follows the same lines as in the previous case, however some estimates need to be changed. Recall \eqref{eq:previous}, namely
\begin{align*}
\left\Vert  u^N_{t} - u_{t}\right\Vert_{\alpha,p}
&\leq \left\Vert  e^{t\mathcal{L}} (u^N_{0} - u_{0}) \right\Vert_{\alpha,p} +  
\mathfrak{I}_t^N + \big\Vert   E_{t} \big\Vert_{\alpha,p} +\left\Vert  \mathcal{M}^N_{t} \right\Vert_{\alpha,p}. 
\end{align*}
 As in \textbf{Case 1}, we bound first $\mathfrak{I}_t^N$ as follows:
\begin{align*}
&\int_0^t  \big\Vert  \nabla  e^{(t-s)\mathcal{L} } \left(u_{s}  \bF(K( u_s)) - u^{N}_{s}  \bF(K(u^{N}_s))\right)\big\Vert_{\alpha,p} ~ \dd s 
\\
&\leq  C  \int_0^t   \frac{1}{(t-s)^{(1+\alpha)/\sigma}} 
\left\Vert  u_s-u^{N}_s \right\Vert_{\LL^p} ~ \dd s 
+ C  \int_0^t   \frac{1}{(t-s)^{(1+\alpha)/\sigma}} \left\Vert u_{s} \right\Vert_{\LL^\infty}
\| K(u_s-u^{N}_s)\|_{\LL^p} ~\dd s 
\\
&\leq  C  \int_0^t   \frac{1}{(t-s)^{(1+\alpha)/\sigma}} 
\left\Vert  u_{s}-u_{s}^{N} \right\Vert_{\alpha, p} ~ \dd s 
 +   C  \int_0^t   \frac{1}{(t-s)^{(1+\alpha)/\sigma}} 
\left\Vert  K(u_{s}-u_{s}^{N}) \right\Vert_{ \LL^p } ~ \dd s 
\end{align*}
where we used as before that $\bF$  is bounded and Lipschitz, and $u$ is bounded. For the remaining two terms, we can start with the same estimates as before, 
\begin{align}
\big\Vert   E_{t} \big\Vert_{\alpha, p }&
%\leq  C \int_0^t  \left\Vert \nabla   e^{(t-s)\mathcal{L}}\  \Big\langle \mu_{s}^N,  V^N (x-\cdot) \Big( \bF\big( K(u^N_s)\big)(x)-\bF\big( K( u^N_s)\big)(\cdot)\Big)\Big\rangle 
%\right\Vert_{\alpha, p } ~ \dd s
%\notag \\ & 
\leq C \int_0^t   \frac{1}{(t-s)^{(1+\alpha)/\sigma}}
\Big\Vert \Big\langle \mu_{s}^N,  V^N (x-\cdot) \ \big| K(u^N_s)(x)- K( u^N_s)(\cdot) \big| \Big\rangle 
\Big\Vert_{ \LL^p }  ~ \dd s
\notag \\ & 
\leq  \frac{C}{N^{\beta(\lambda-d/p) }} \int_0^t    \frac{1}{(t-s)^{(1+\alpha)/\sigma}} \
\left\Vert \langle \mu_{s}^N,  V^N (x-\cdot)  \rangle 
\right\Vert_{\LL^p}  \left(\Vert  u_s^{N}\Vert_{\LL^p} +1\right)      ~\dd s
\notag \\ & 
\leq  \frac{C}{N^{\beta(\lambda-d/p) }} \int_0^t 
\frac{1}{(t-s)^{(1+\alpha)/\sigma}} \
  \left( \Vert  u_{s}^{N} \Vert_{\LL^p}^{2} +   \left\Vert  u_{s}^{N} \right\Vert_{\LL^p}\right)  \dd s,  \label{eq:estiE}
\end{align}
where we  used the fact that $\mathbf{F}$  is   Lipschitz continuous, the second assumption \eqref{eq:dCK2} on the kernel operator, 
the Sobolev embedding \eqref{eq:SE1} and the fact that $\|u_{s}^{N}\|_{\LL^1}=1$. 

As previously, we now take expectations and use the intermediate lemmas, namely: Lemma \ref{estimation} with $\kappa=0$, in order to bound the expectation of the right-hand side of \eqref{eq:estiE} and   Lemma \ref{lemma martingale 2}, in order to estimate $\EE[\Vert  \mathcal{M}^N_{t}\Vert_{\alpha,p}]$.  
We deduce  that 
\begin{align}
\EE \big[\left\Vert  u^N_{t} - u_{t} \right\Vert_{\alpha,p}\big]
&\leq \EE \big[ \left\Vert  u^N_{0} - u_{0} \right\Vert_{\alpha,p}\big] +  
 C  \int_0^t   \frac{1}{(t-s)^{(1+\alpha)/\sigma}}  \notag
\EE \big[\left\Vert  u_{s}-u_{s}^{N} \right\Vert_{\alpha,p}\big] ~ \dd s  \\
& \qquad +  C  \int_0^t   \frac{1}{(t-s)^{(1+\alpha)/\sigma}} 
\EE \big[ \| K(u_{s}-u_{s}^{N}) \|_{\LL^p}\big] ~ \dd s \label{eq:esK} \\
&\qquad +  \frac{C}{N^{\beta(\lambda-d/p) }}  +  \frac{C}{N^{\frac12(1-  2\beta( \alpha +d -d/p + \delta ))}} . \notag 
\end{align}
Now, on the contrary to the previous case, there is little more work to close this estimate, in particular because of \eqref{eq:esK}. Applying $K$ to the mild formula \eqref{eq:mild2bis}, let us bound similarly:
\begin{align*}
\left\Vert  K(u^N_{t}- u_{t}) \right\Vert_{\LL^p}
&\leq \left\Vert  e^{t\mathcal{L}}K (u^N_{0} - u_{0}) \right\Vert_{\LL^p} \vphantom{\int}\\ & \qquad +  
\int_0^t  \left\Vert  \nabla  K  e^{(t-s)\mathcal{L} } \left(u_{s}  \bF(K( u_s)) - u^{N}_{s}  \bF(K(u^{N}_s))\right)\right\Vert_{\LL^p} ~ \dd s \\
& \vphantom{\int} \qquad + \big\Vert   K E_{t} \big\Vert_{\LL^p} +\left\Vert  K \mathcal{M}^N_{t} \right\Vert_{\LL^p}. 
\end{align*}
We first treat the second term in the right-hand side: we obtain, by triangular inequality and using assumption \eqref{eq:dCK2}, 
\begin{align*}
\int_0^t  &\left\Vert  \nabla K  e^{(t-s)\mathcal{L} } \left(u_{s}  \bF(K( u_s)) - u^{N}_{s}  \bF(K(u^{N}_s))\right)\right\Vert_{\LL^p} ~ \dd s \\ &
\leq  C  \int_0^t   \left\Vert  (u_{s}-u_{s}^{N})  \bF(K(u^{N}_s))  \right\Vert_{\LL^p} ~ \dd s 
+ C  \int_0^t    \left\Vert   u_{s} \big( \bF(K( u_s)) - \bF(K(u^{N}_s))\big) \right\Vert_{\LL^p} ~ \dd s 
\\&
\leq  C  \int_0^t   \left\Vert  u_s-u^{N}_s \right\Vert_{\LL^p} ~ \dd s 
+ C  \int_0^t   \left\Vert u_{s} \right\Vert_{\LL^\infty}
\| K(u_s-u^{N}_s)\|_{\LL^p}~ \dd s 
\end{align*}
where in the last inequality we used once again the fact that $\bF$  is bounded and Lipschitz continuous, and also that $u$ is bounded from Proposition \ref{propu}. Let us now estimate, using very similar ideas as in \eqref{eq:estiE}:
\begin{align*}
\big\Vert  K E_{t} \big\Vert_{ \LL^p }&
\leq  C \int_0^t  \left\Vert \nabla K  \Big\langle \mu_{s}^N,  V^N (x-\cdot) \Big( \bF\big( K(u^N_s)\big)(x)-\bF\big( K( u^N_s)\big)(\cdot)\Big)\Big\rangle 
\right\Vert_{\LL^p } ~ \dd s
\\ & 
\leq C \int_0^t   
\Big\Vert \Big\langle \mu_{s}^N,  V^N (x-\cdot) \ \big| K(u^N_s)(x)- K( u^N_s)(\cdot) \big| \Big\rangle 
\Big\Vert_{ \LL^p }  ~ \dd s
\\ & 
\leq  \frac{C}{N^{\beta(\lambda-d/p) }} \int_0^t   
\left\Vert \langle \mu_{s}^N,  V^N (x-\cdot)  \rangle 
\right\Vert_{\LL^p}  \left(\Vert  u_s^{N}\Vert_{\LL^p} +1\right)      ~\dd s
\\ & 
\leq  \frac{C}{N^{\beta(\lambda-d/p) }} \int_0^t   
  \left( \Vert  u_{s}^{N} \Vert_{\LL^p}^{2} +   \left\Vert  u_{s}^{N} \right\Vert_{\LL^p}\right)  \dd s, 
\end{align*}
where we  used, as before the fact that $\mathbf{F}$  is   Lipschitz continuous, assumption \eqref{eq:dCK2}, 
the Sobolev embedding \eqref{eq:SE1} and the fact that $\|u_{s}^{N}\|_{\LL^1}=1$.  After taking expectation we use once again Lemma \ref{estimation}, and we obtain: 
\[ \EE\big[\|KE_t\|_{\LL^p}\big] \leqslant\frac{C}{N^{\beta(\lambda-d/p)}}. \]
It remains to treat the term involing the martingale. Let us write  
\begin{align} 
\left\Vert  K \mathcal{M}^N_{t} \right\Vert_{\LL^p} 
&=\left\Vert (-\Delta)^{1/2} (-\Delta)^{-1/2} K \mathcal{M}^N_{t} \right\Vert_{\LL^p}  \notag
\\ & 
=\left\Vert (-\Delta)^{1/2}  K (-\Delta)^{-1/2} \mathcal{M}^N_{t} \right\Vert_{\LL^p} \notag
\\ & 
=| K (-\Delta)^{-1/2} \mathcal{M}^N_{t} |_{1,p}  \notag
\\ & 
\leq C \left\Vert \nabla K (-\Delta)^{-1/2} \mathcal{M}^N_{t} \right\Vert_{\LL^p}  \notag
\\ & 
\leq  C\left\Vert (-\Delta)^{-1/2} \mathcal{M}^N_{t} \right\Vert_{\LL^p} \notag \end{align}
where we used in the last inequality the assumption \eqref{eq:dCK2}. We finally use a third intermediate lemma, similar to the previous one, and also proved in Appendix \ref{sec:app}: 
\begin{lemma} \label{lemma martingale 4}  Take any $\kappa$ such that $-1 \leq \kappa \leq \alpha < \sigma -1$, recalling that Assumption \ref{assump} is verified.  Then, there exists a constant $C>0$, such that,  for any  $N\in\mathbb N$ and $t\geq 0$ it holds,
\[
\EE \Big[\big\Vert  \left( -\Delta\right)^{\kappa/2} \mathcal{M}_t^{N}    \big\Vert _{\LL^{p}(\mathbb{R}^{d})}^{2} \Big]\leq
\frac{C}{N^{  1- 2\beta( \kappa +d -d/p + \delta )    }} . 
\]
\end{lemma}
Therefore, taking expectation and using the above lemma with $\kappa=-1$,  we obtain\begin{align}
\EE\big[\left\Vert  K \mathcal{M}^N_{t} \right\Vert_{\LL^p} \big]\leq \frac{C}{N^{\frac12(1- 2 \beta( -1 +d -d/p + \delta ))}}
\end{align}
Summarizing, after putting all estimates together we have 
\begin{align}
\EE \big[\left\Vert  K(u^N_{t} - u_{t}) \right\Vert_{\LL^p}\big]
&\leq \EE\big[ \left\Vert  K(u^N_{0} - u_{0}) \right\Vert_{\LL^p}\big] +  
 C  \int_0^t    \EE\big[ \left\Vert  u_{s}-u_{s}^{N} \right\Vert_{\LL^p}\big] ~ \dd s \notag \\
& \quad +  C  \int_0^t   
\EE\big[ \| K(u_{s}-u_{s}^{N}) \|_{\LL^p} \big]~ \dd s  \label{eq:esK2}\\ & 
\quad +  \frac{C}{N^{\beta(\lambda-d/p) }}  +  \frac{C}{N^{\frac12(1-  2\beta( -1 +d -d/p + \delta ))}}\notag. 
\end{align}
Therefore, summing (\ref{eq:esK}) with  (\ref{eq:esK2}), and recalling the definition of the distorted norm $\|\cdot\|_{\tilde{\cX}}$, we get
\begin{align}
\EE \big[\left\Vert  u^N_{t} - u_{t} \right\Vert_{\tilde{\mathcal{X}}}\big]
&\leq \EE\big[\left\Vert  u^N_{0} - u_{0} \right\Vert_{\tilde{\mathcal{X}}}\big] +  
 C  \int_0^t   \bigg(\frac{1}{(t-s)^{(1+\alpha)/\sigma}} +1\bigg)
\EE \big[\left\Vert  u_{s}-u_{s}^{N} \right\Vert_{\tilde{\mathcal{X}}}\big] ~ \dd s \notag \\
& \quad +  \frac{C}{N^{\beta(\lambda-d/p) }}  +  \frac{C}{N^{\frac12(1-  2\beta( \alpha +d -d/p + \delta ))}} . 
\end{align}
Applying Gronwall's Lemma we have proved Theorem \ref{th:rate2sing} in \textbf{Case 2}. \qed

\section{Proof of Corollary \ref{Colo}}
\label{sec:proof2}
Let $q\geq 1$ and $\varepsilon >0$. 
Recall that $p_{\epsilon}$  is  the conjugate of $q-\epsilon$, that is $\frac{1}{p_{\epsilon}} + \frac{1}{q-\epsilon}=1 ,$ and recall the definition of $r_\varepsilon$ given in \eqref{eq:defr}. 
We take $u_{0}\in \HH_{q-\epsilon}^{-r_{\epsilon}}$, and we shall prove that  
$u\in C([0,T],  \HH_{q-\epsilon}^{-r_{\epsilon}})$. 

First let us show that 
$ K(u)\in C([0,T],  \HH_{p_{\epsilon}}^{r_{\epsilon}})$.   From the mild formulation and triangular inequality we obtain 
\begin{align}
\left\Vert K(u_{t}) \right\Vert_{r_{\epsilon},p_{\epsilon}}
&\leq \left\Vert K( e^{t\mathcal{L}}  u_{0} ) \right\Vert_{r_{\epsilon},p_{\epsilon}} \vphantom{\int_0^1} \notag +
\int_0^t  \left\Vert K  \nabla e^{(t-s)\mathcal{L} } u_{s}  K( u_s)  \right\Vert_{r_{\epsilon},p_{\epsilon}} ~ \dd s  . \label{milK}
\end{align}
By the assumption of Corollary \ref{Colo}, for any $t\geqslant 0$, we have
\[
\left\Vert K( e^{t\mathcal{L}}  u_{0} ) \right\Vert_{r_{\epsilon},p_{\epsilon}}\leq C.
\]
Moreover, from the  semigroup property \eqref{eq:semigroup} we may write
\begin{align*}
\int_0^t  \left\Vert K  \nabla e^{(t-s)\mathcal{L} } u_{s}  K( u_s)  \right\Vert_{r_{\epsilon},p_{\epsilon}} ~ \dd s &
\leq \int_0^t  \frac{1}{(t-s)^{ d/\sigma  (1/p_{\epsilon}- 1/p)}} \left\Vert \nabla K ( u_{s}  K( u_s))  \right\Vert_{r_{\epsilon},p} ~ \dd s.
\end{align*} 
We observe that from our assumption, $ d/\sigma  (1/p_{\epsilon}- 1/p)<1$. Now,   since $\alpha> r_{\epsilon}$, 
$\HH_{p}^{\alpha}$ is an algebra. Using the hypothesis \eqref{eq:dCK}  and (\ref{eq:dCK2}) we obtain 
\begin{align*}
\int_0^t  \frac{1}{(t-s)^{ d/\sigma  (1/p_{\epsilon}- 1/p)}}& \left\Vert \nabla K ( u_{s}  K( u_s))  \right\Vert_{r_{\epsilon},p} ~ \dd s \\
& 
\leq  \int_0^t  \frac{1}{(t-s)^{ d/\sigma  (1/p_{\epsilon}- 1/p)}} \left\Vert \nabla K ( u_{s}  K( u_s))  \right\Vert_{\alpha,p} ~ \dd s \\
& 
= \int_0^t  \frac{1}{(t-s)^{  d/\sigma  (1/p_{\epsilon}- 1/p)}} \left\Vert  \nabla K 
(I- A)^{\alpha/2}( u_{s}  K( u_s))  \right\Vert_{\LL^p} ~ \dd s
\\
& \leq  C
 \int_0^t  \frac{1}{(t-s)^{  d/\sigma  (1/p_{\epsilon}- 1/p)}} \left\Vert  
(I- A)^{\alpha/2}( u_{s}  K( u_s))  \right\Vert_{\LL^p} ~ \dd s
\\ & =  C
 \int_0^t  \frac{1}{(t-s)^{  d/\sigma  (1/p_{\epsilon}- 1/p)}} \left\Vert  
 u_{s}  K( u_s)  \right\Vert_{\alpha,p} ~ \dd s
\\  
& \leq 
\int_0^t  \frac{1}{(t-s)^{  d/\sigma  (1/p_{\epsilon}- 1/p)}} \left\Vert   u_{s} \right\Vert_{\alpha,p} \left\Vert K( u_s)  \right\Vert_{\alpha,p} ~ \dd s
\\
& \leq   \int_0^t  \frac{1}{(t-s)^{ d/ \sigma (1/p_{\epsilon}- 1/p)}} \left\Vert   u_{s}   \right\Vert_{\alpha,p}^{2} ~ \dd s.
\end{align*}
From this we conclude  $ K(u)\in C([0,T],  \HH_{p_{\epsilon}}^{r_{\epsilon}})$.
\medskip

Now, let us prove that $u\in C([0,T],\HH_{q-\varepsilon}^{-r_\varepsilon})$.
Once again, from the mild formulation and triangular inequality we obtain 
\begin{align}
\left\Vert u_{t} \right\Vert_{-r_{\epsilon},q-\epsilon}
&\leq \left\Vert  e^{t\mathcal{L}}  u_{0} \right\Vert_{-r_{\epsilon},q-\epsilon} \vphantom{\int_0^1} \notag +
\int_0^t  \left\Vert  \nabla  e^{(t-s)\mathcal{L} } u_{s}  K( u_s)  \right\Vert_{-r_{\epsilon},q-\epsilon} ~ \dd s .  \label{milnega}
\end{align}
We also have 
\[
\left\Vert  e^{t\mathcal{L}}  u_{0} \right\Vert_{-r_{\epsilon},p}\leq
\left\Vert  u_{0} \right\Vert_{-r_{\epsilon},p}\leq C.
\]
From the semigroup property,  by duality  we deduce 
\begin{align*}
\int_0^t  \left\Vert  \nabla  e^{(t-s)\mathcal{L} } u_{s}  K( u_s)  \right\Vert_{-r_{\epsilon},q-\epsilon} ~ \dd s \notag
&
\leq \int_0^t   \frac{1}{(t-s)^{1/\sigma}}     \left\Vert u_{s}  K( u_s)  \right\Vert_{-r_{\epsilon},q-\epsilon} ~ \dd s \notag
\\ & 
\leq \int_0^t   \frac{1}{(t-s)^{1/\sigma}}     \left\Vert u_{s}  \right\Vert_{-r_{\epsilon},q-\epsilon}    \left\Vert K( u_s)  \right\Vert_{r_{\epsilon},p_{\epsilon}} ~ \dd s \notag.
\end{align*}
Therefore
\begin{align*}
\left\Vert u_{t} \right\Vert_{-r_{\epsilon},q-\epsilon}
&\leq   C  +
\int_0^t  \frac{1}{(t-s)^{1/\sigma}}     \left\Vert u_{s}  \right\Vert_{-r_{\epsilon},q-\epsilon}    \left\Vert K( u_s)  \right\Vert_{r_{\epsilon},p_{\epsilon}} ~ \dd s  \\ 
&\leq C + \int_0^t  \frac{1}{(t-s)^{1/\sigma}}     \left\Vert u_{s}  \right\Vert_{-r_{\epsilon},q-\epsilon}    ~ \dd s.
\end{align*}
From Gronwall's Lemma we conclude $u\in C([0,T],  \HH_{q-\epsilon}^{-r_{\epsilon}})$.
Now, we will prove that
\[
\sup_{t\in [0,T]}\|  \mu_{t}^N  \|_{-r_{\epsilon}, q-\epsilon}\leq C.
\]
This holds because
\begin{align*}
\|  \mu_{t}^N  \|_{-r_{\epsilon}, q-\epsilon}&= \sup_{\| \varphi  \|_{ r_{\epsilon},p_{\epsilon}} \leq 1 } |\langle \mu_{t}^N, \varphi \rangle |
\\
&\leq  \sup_{\| \varphi  \|_{ r_{\epsilon},p_{\epsilon}} \leq 1 } \sup_{x\in\RR^{d}} |\varphi(x)|
\\ &
\leq   C\sup_{\| \varphi  \|_{ r_{\epsilon},p_{\epsilon}} \leq 1 }   \| \varphi  \|_{ r_{\epsilon},p_{\epsilon}}=C
\end{align*}
 where we used the Sobolev embedding \eqref{eq:SE1} (recall that we assume $r_{\epsilon}= \alpha(2\theta_{\epsilon} -1)> d/{p_{\epsilon}}$). We are now ready to conclude the proof of Corollary \ref{Colo}.
 
 Applying  the triangular inequality we have
\begin{align*}
   \sup_{t\in [0,T]}\EE \big[ \| \mu_{t}^N-u_t \|_{-\alpha ,q}\big]&\leq
	  \sup_{t\in [0,T]}\EE \big[\| \mu_{t}^N-u_t^{N} \|_{-\alpha,q} \big] +
		\sup_{t\in [0,T]}\EE \big[ \| u_t^N-u_t \|_{-\alpha, q}\big]
\\ & 
	=: {\rm I}_{1} + {\rm I}_{2}. 
\end{align*}
We first estimate ${\rm I}_{2}$, by interpolation  and using  Theorem \ref{th:rate2sing}
we deduce 
\begin{align*}
 {\rm I}_{2}&\leq     \sup_{t\in [0,T]}\EE\big[ \|  u_t^N-u_t \|_{-r_{\epsilon}, q-\epsilon}^{\theta_{\epsilon}}\big] \; 
\sup_{t\in [0,T]}\EE \big[\| u_t^N-u_t \|_{\alpha, p}^{1- \theta_{\epsilon}}\big]
\\ &
\leq C \sup_{t\in [0,T]}\EE\big[ \| u_t^N-u_t \|_{\alpha, p}^{1- \theta_{\epsilon}}\big]=
C_{T} \sup_{t\in [0,T]} \EE \big[\|u^N_0- u_0 \|_{\tilde{\mathcal{X}}}^{1-\theta_{\epsilon}} \big] + 
 \frac{C}{N^{\rho(1- \theta_{\epsilon})}}. 
\end{align*}
Now, we estimate ${\rm I}_{1}$, we observe 
\begin{align*}
  \| \mu_{t}^N-u_{t}^{N} \|_{-\alpha,q}&= \sup_{\|\varphi \|_{\alpha,p} \leq 1} |\langle \mu_{t}^N, (V^{N}\ast\varphi)(.)- \varphi(.)  \rangle|
\\ & 
\leq \sup_{\|\varphi \|_{\alpha,p} \leq 1} \sup_{x\in \RR^{d}}| (V^{N}\ast\varphi)(x)- \varphi(x)| 
	\\ & 
\leq  \frac{C}{N^{\beta(\alpha -d/p )}}, 
	\end{align*}
where we used  the Sobolev embedding \eqref{eq:SE1} since $\alpha > \frac dp$.  This ends the proof. \qed

\appendix 
\section{Intermediate lemmas}
\label{sec:app}

In this appendix we prove the three intermediate lemmas used in the main argument.

\subsection{Uniform bounds for  $u^{N}$: proof of Lemma \ref{estimation}}
\label{sec:bound}

We recall and prove the following:

\begin{lemma}\label{estimation2} Take any $\kappa$ such that $-1 \leq \kappa\leq \alpha < \sigma-1$, recalling that Assumption \ref{assump} is verified. There exists a constant $C>0$ such that, for any $t\in[0,T]$ and $N\in\mathbb{N}$, it holds
\[
\mathbb{E}\big[  \big\Vert \left(  I-A \right) ^{\kappa/2}u_{t}
^{N}\big\Vert^2_{\LL^{p}}\big]  \leq
C.
\]
\end{lemma}

\begin{proof} We denote $\mathfrak{X}=\LL^{p}(\mathbb{R}^{d})$. From
 the mild formulation (\ref{eq:mild2})  we  obtain 
\begin{align}
\big\Vert \left(  I-A\right)^{\kappa/2}&u
_{t}^{N}\big\Vert _{\LL^{2}(  \Omega, \mathfrak{X})  } \notag \vphantom{\int}\\ &
\leq\big\Vert \left(  I-A\right)  ^{\kappa/2}e^{t\mathcal{L}%
}u_{0}^{N}\big\Vert _{\LL^{2}(  \Omega, \mathfrak{X})  } \label{eq:1} \vphantom{\int}\\
&  \qquad +\int_{0}^{t}\big\Vert \left(  I-A\right)  ^{\kappa/2}\nabla
e^{\left(  t-s\right)\mathcal{L}}\left(   V^{N} \ast\left(
\bF( K(u_{s}^{N}))\mu_{s}^{N}\right)  \right)  \big\Vert _{\LL^{2}(
\Omega, \mathfrak{X} )  }~ \dd s \label{eq:2}\\
& \qquad + \big\Vert  \left(  I-A\right)^{\kappa/2} \mathcal{M}_t^{N}    \big\Vert _{\LL^{2}(
\Omega, \mathfrak{X})  } \label{eq:3} \vphantom{\int}
\end{align}
where $\mathcal{M}_t^{N}$ has been defined in \eqref{eq:martingale}.

The first term \eqref{eq:1} can be estimated by (since $\kappa \leq \alpha$) \[ \big\Vert \left(  I-A\right) ^{\kappa/2} u_{0}^{N}\big\Vert _{\LL^{2}(  \Omega, \mathfrak{X}) } \leq \big\Vert \left(  I-A\right) ^{\alpha/2} u_{0}^{N}\big\Vert _{\LL^{2}(  \Omega, \mathfrak{X}) }  \]
which is bounded by a constant $C>0$  from Assumption \ref{assump}, see \eqref{initial cond}. 

Let us bound the second term \eqref{eq:2} as follows:
%We have%
%{\color{red}
%\begin{align*}
%e^{\left(  t+h-s\right)  L}\left(  \nabla V^{N}\ast\left(  F(\cdot\; ,\; g_{t}%
%^{N})S_{s}^{N}\right)  \right)   &  =\sum_{j=1}^{d}e^{\left(  t+h-s\right)
%L}\partial_{j}\left(  V_{N}\ast\left(  F_{j}(\cdot\; ,\;g_{t}^{N})S_{s}^{N}\right)
%\right) \\
%&  =\sum_{j=1}^{d}\partial_{j}e^{\left(  t+h-s\right)  L}\left(  V^{N}%
%\ast\left(  F_{j}(\cdot\; ,\;g_{t}^{N})S_{s}^{N}\right)  \right) \\
%&  =:\nabla e^{\left(  t+h-s\right)  L}\left(  V^{N}\ast\left(
%F(\cdot\; ,\;g_{t}^{N})S_{s}^{N}\right)  \right)
%\end{align*}
%using the convolution formulation of the action of $e^{\left(  t+h-s\right)L}$ 
%
%**** this is useless, right? this is already written in the previous paragraph... ****}. Hence
\begin{multline*}
 \int_{0}^{t}\big\Vert \left(  I-A\right)^{\kappa/2}\nabla
e^{\left(  t-s\right) \mathcal{L}}\left(  V^{N} \ast\left(
\bF( K(u_{s}^{N}))\mu_{s}^{N}\right)  \right)  \big\Vert _{\LL^{2}(
\Omega, \mathfrak{X})  }~ \dd s\\
    \leq  C \ \int_{0}^{t}\big\Vert (  I-A  )^{(1+\kappa)/2}
e^{\left(  t-s\right) \mathcal{L}}\big\Vert _{\LL^{p}
\rightarrow \LL^{p}  }\ \big\Vert  
 V^{N} \ast\left(  \bF( K(u_{s}^{N}))\mu_{s}^{N}\right)
  \big\Vert _{\LL^{2}(  \Omega,  \mathfrak{X} )}~ \dd s.
\end{multline*}
 We have, from Proposition \ref{prop} applied with exponent $(1+\kappa)/2\geq 0$, that
\begin{equation*}
 \big\Vert \left(  I-A\right)^{(1+\kappa)/2}e^{\left(  t-s\right)\mathcal{L}} \big\Vert _{\LL^{p}  \rightarrow \LL^{p} } \leq
 \frac{C_{\kappa,\sigma,T,p}}{(t-s)^{ (1+\kappa)/\sigma}}.
\end{equation*}
Therefore, since $\bF$ is bounded, there is $C>0$ such that \eqref{eq:2} is bounded by 
\[
\int_{0}^{t}\frac{C}{\left(  t-s\right) ^{(1+\kappa)/\sigma}}\big\Vert  u_{s}^{N}\big\Vert _{\LL^{2}(
\Omega,  \mathfrak{X}) }~ \dd s.
\]
The estimate of the third term \eqref{eq:3} is quite tricky and we
postpone it to Lemma   \ref{lemma martingale 2} below. Collecting the three
bounds together, we have%
\[
\big\Vert \left(  I-A\right) ^{\kappa/2}u%
_{t}^{N}\big\Vert _{\LL^{2}(  \Omega,  \mathfrak{X})  }%
\leq C  +\int_{0}^{t}\frac{C}{\left(  t-s\right)
^{(1+\kappa) /\sigma}}\big\Vert u_{s}^{N}\big\Vert _{\LL^{2}(
\Omega,  \mathfrak{X})}~ \dd s.
\]
We may apply  Gronwall's Lemma, and since $1+\kappa \leq 1+\alpha < \sigma$ from Assumption \ref{assump}, we conclude 
\[
\big\Vert \left(  I-A\right)^{\kappa/2}u_{t}^{N}\big\Vert _{\LL^{2}(  \Omega,  \mathfrak{X})  }%
\leq C
\]
as wanted.
\end{proof}

%\begin{remark}
%Taking $\alpha=0$ we get that $g_t^N$ is also uniformly bounded: for any $t \in [0,T]$ and $N \in \mathbb{N}$, 
%\begin{equation}
%\label{eq:L2g} \mathbb{E}\big[ \Vert g_t^N \Vert_{\LL^p}^{q}\big] \leq C.
%\end{equation}
%\end{remark}

\subsection{Martingale estimates}

We recall and prove the second intermediate lemma: 

\begin{lemma} \label{lemma martingale 2bis} Take any $\kappa$ such that $-1 \leq \kappa \leq \alpha < \sigma -1$, recalling that Assumption \ref{assump} is verified.  Then, there exists a constant $C>0$, such that,  for any  $N\in\mathbb N$ and $t\geq 0$ it holds,
\[
\EE \Big[\big\Vert  \left(  I-A\right)^{\kappa/2} \mathcal{M}_t^{N}    \big\Vert _{\LL^{p}}^{2} \Big]\leq
\frac{C}{N^{  1-  2\beta( \kappa +d -d/p + \delta )     }} . 
\]
\end{lemma}

\begin{proof}  From the Sobolev embedding \eqref{eq:SE2}, applied with $q=2<p$ by hypothesis (point \emph{2.} of Assumption \ref{assump}), we have 
\[
\EE \Big[ \big\Vert  \left(  I-A\right)^{\kappa/2} \mathcal{M}_t^{N}    \big\Vert _{\LL^{p}}^{2} \Big]
\leq C \EE \Big[\big\Vert  \left(  I-A\right)^{(\kappa+ d/2 -d/p)/2} \mathcal{M}_t^{N}    \big\Vert _{\LL^{2}}^{2}\Big]. 
 \]
For the sake of simplicity we introduce the notation: $\eta:=\frac12(\kappa + \frac d 2 - \frac d p)$.  Recall the definition of $\mathcal{M}_t^N$ in \eqref{eq:martingale}, and let us decompose $\mathcal{M}_t^N$ according to the sum of two integrals, over $|z|\leq 1$ and $|z| \geq 1$, as follows: 
\begin{align*}
& \mathcal{M}_t^{N}
 = {\rm I}_1 + {\rm I}_2 \\
%&  \frac{1}{N}\sum_{i=1}^{N} \int_{0}^{t} \int_{\RR^{d}-\{0\}}   \big(
%e^{(t-s+ h)L} V^{N} \big)  \big(  x-X_{s_{-}}^{i,N} + z  \big)  -   
%   \big(e^{(t-s+ h)\mathcal{L}} V^{N} \big)\big( x-X_{s_{-}}^{i,N}\big)  d\tilde{\mathcal N}^{i}(dsdz)\\
 & =: \frac{1}{N}\sum_{i=1}^{N} \int_{0}^{t} \int_{|z|\geq 1}  \Big\{ \big(
e^{(t-s)\mathcal{L}} V^{N} \big)  \big(  x-X_{s_{-}}^{i,N} + z  \big)  -   
   \big(e^{(t-s)\mathcal{L}} V^{N}\big)\big( x-X_{s_{-}}^{i,N}\big) \Big\} \tilde{\mathcal N}^{i}(\dd s\dd z)\\
& +   \frac{1}{N}\sum_{i=1}^{N} \int_{0}^{t} \int_{|z|\leq1}  \Big\{ \big(
e^{(t-s)\mathcal{L}} V^{N} \big)  \big(  x-X_{s_{-}}^{i,N} + z  \big)  -   
   \big(e^{(t-s)\mathcal{L}} V^{N} \big)\big( x-X_{s_{-}}^{i,N}\big)\Big\}  \tilde{\mathcal N}^{i}(\dd s\dd z). 
\end{align*}
The reason for this decomposition is the following: when $|z|\geq 1$ we can use the fact that $\int_{|z|\geq 1} \dd\nu(z) < \infty$ ; while when $|z|\leq 1$  we know that $\int_{|z|\leq 1} |z|^2 \dd\nu(z) < \infty$.  For the sake of clarity, along this proof we denote
\begin{equation}\label{eq:fn}
 \mathcal{F}^N(s,x,z) := V^{N}  \big(  x-X_{s_{-}}^{i,N} + z  \big)  -   
   V^{N}\big( x-X_{s_{-}}^{i,N}\big).
\end{equation}

\medskip

\noindent 
\textsc{Step 1}.  Let us start with ${\rm I}_1$. We use directly the $\LL^2$ norm in the product space $\Omega \times \RR^d$: we have 
\begin{align*}
  \big\Vert (  I&-A)^{\eta}  \; {\rm I}_1\big\Vert _{\LL^{2}(  \Omega\times
\mathbb{R}^{d})  }^{2}\\
&  =
\frac{2}{N^{2}} \int_{\mathbb{R}^{d}}\mathbb{E}\Bigg[  \Bigg\vert \sum_{i=1}^{N} \int_{0}^{t}
\int_{|z|\geq 1}  \left(  I-A\right)^{\eta} e^{(t-s)\mathcal{L}}  \mathcal{F}^N(s,x,z) 
 \tilde{\mathcal N}^{i}(\dd s\dd z)   \Bigg\vert ^{2}\Bigg]  \dd x\\
&  =\frac{2}{N^{2}}     \sum_{i=1}^{N} \int_{\mathbb{R}^{d}}\mathbb{E}\left[  \int_{0}^{t}
\int_{|z|\geq 1} \Big| \left(  I-A\right)^{\eta}  e^{(t-s)\mathcal{L}}  \mathcal{F}^N(s,x,z) \Big|^{2}
  \dd\nu(z) \dd s    \right]  \dd x \\
&  = \frac{2}{N}  \;   \mathbb{E} \left[  \int_{0}^{t} \int_{|z|\geq 1}  \int_{\mathbb{R}^{d}}
 \Big| \left(  I-A\right)^{\eta}  e^{(t-s)\mathcal{L}}  \mathcal{F}^N(s,x,z) \Big|^{2} \ \dd x      \dd\nu(z) \dd s    \right].  
\end{align*}
 We observe  that, by triangular inequality 
\begin{align*}
 \int_{\mathbb{R}^{d}}
 \Big| \left(  I-A\right)^{\eta}  e^{(t-s)\mathcal{L}}  &\mathcal{F}^N(s,x,z) \Big|^{2} \ \dd x \\ &\leq 2 \int_{\mathbb{R}^{d}}
 \Big| \left(  I-A\right)^{\eta}  e^{(t-s)\mathcal{L}}  V^{N}(x) \Big|^{2} \ \dd x  \leq  C \big\|  V^{N} \big\|_{\HH^{2\eta}_2}^{2}. \vphantom{\int}
\end{align*}
This implies that
\begin{align*}
 \big\Vert \left(  I-A\right)^{\eta} \; {\rm I}_1\big\Vert _{\LL^{2}(  \Omega\times
\mathbb{R}^{d})  }^{2}&\leq
\frac{C}{N}  \big\| V^{N} \big\|_{\HH^{2\eta}_2}^{2} 
 \leq
\frac{C}N   N^{\beta d + 2\beta \eta }\; \|V\|_{\HH^{2\eta}_2}^{2} ,
\end{align*}
where we used a change of variables in the last inequality.

Recall that $2\eta = \alpha+\frac d 2 - \frac d p$, and  $ \|V\|_{\HH^{\alpha+ d/2 -d/p}_2} < \infty$ since $V$ is smooth compactly supported (Assumption \ref{assump}). Therefore we finally get
\[   \big\Vert \left(  I-A\right)^{\eta} \; {\rm I}_1\big\Vert _{\LL^{2}(  \Omega\times
\mathbb{R}^{d})  }^{2} \leqslant \frac{C}{N^{1-\beta (2d + 2\kappa -2d/p)}}.\]

\noindent \textsc{Step 2}. In the same way, for the second part of the integral we obtain 
\begin{multline*}
 \big\Vert (  I-A)^{  \eta}  \; {\rm I}_2\big\Vert _{\LL^{2}(  \Omega\times
\mathbb{R}^{d})  }^{2} \\ = \frac{2}{N^{2}}     \sum_{i=1}^{N}\mathbb{E} \left[  \int_{0}^{t} \int_{|z|\leq  1}  \int_{\mathbb{R}^{d}}
 \big| \left(  I-A\right)^{\eta}  e^{(t-s)\mathcal{L}}  \mathcal{F}^N(s,x,z) \big|^{2} \ \dd x    \dd\nu(z) \dd s    \right]  
\end{multline*}
where $ \mathcal{F}^N$ has been defined in \eqref{eq:fn}. Here,  we use the tool of the maximal function introduced in Section \ref{sec:max}.
From Lemma \ref{maxi} (with $f=V$ which is smooth and compactly supported), and triangular inequality, we get
\begin{align*}
(  I-A)^{\eta}   e^{(t-s)\mathcal{L}}& \mathcal{F}^N(s,x,z)  \\
\leq &\;  C |z|\; \mathbb{M} \Big[ \nabla(I-A)^{ \eta} \; e^{(t-s)\mathcal{L}} V^N\big(x-X_{s_{-}}^{i,N} + z \big)\Big]\\
 & + C |z| \; \mathbb{M}\Big[ \nabla(I-A)^{\eta} \; e^{(t-s)\mathcal{L}} V^N\big(x-X_{s_{-}}^{i,N}\big) \Big].
\end{align*}
Moreover, again from  Lemma \ref{maxi} we have
\[
\Big\|  \mathbb{M} \Big[\nabla(I-A)^{\eta}  e^{(t-s)\mathcal{L}} V^{N} \Big] \Big\|_{\LL^{2}}^{2}\leq C
\Big\|\nabla  (I-A)^{\eta}  e^{(t-s)\mathcal{L}} V^{N} \Big\|_{\LL^{2}}^{2}.
\]
This implies that 
\[
\Big\| (  I-A)^{\eta}  e^{(t-s)\mathcal{L}} \mathcal{F}^N(s,x,z)  \Big\|_{\LL^2}^{2} \leq C \; |z|^2
\; \Big\|\nabla  (I-A)^{\eta}  e^{(t-s)\mathcal{L}} V^{N} \Big\|_{\LL^{2}}^{2}.
\]
Taking $\delta> 1-\frac \sigma 2 $ as in point \textit{4.} of Assumption \ref{assump}, we finally get
\begin{align*}
 \big\Vert (  I-A)^{ \eta } \;  {\rm I}_2\big\Vert _{\LL^{2}(  \Omega\times
\mathbb{R}^{d})  }^{2}&  \leq    \frac{C}{N}  \int_{0}^{t} \big\|\nabla  (I-A)^{ \eta}  e^{(t-s)\mathcal{L}} V^{N} \big\|_{\LL^{2}}^{2}
 ~ \dd s  \\
&  \leq    \frac{C}{N}   \big\| (I-A)^{\eta + \frac12 \delta} \; V^{N}\big\|_{\LL^{2}}^{2} 
 \int_{0}^{t} \frac{1}{(t-s)^{2(1-\delta)/ \sigma}} \dd s 
\\ & \leq  \frac{C}{N^{1- d\beta- 2 (\kappa+ d/2 -d/p +\delta)\beta  }},
\end{align*}
from the same computations as before. This concludes the proof.
\end{proof}

\begin{lemma} \label{lemma martingale 4bis}  Take any $\kappa$ such that $-1 \leq \kappa \leq \alpha < \sigma -1$, recalling that Assumption \ref{assump} is verified.  Then, there exists a constant $C>0$, such that,  for any  $N\in\mathbb N$ and $t\geq 0$ it holds,
\[
\EE \Big[\big\Vert  \left( -\Delta\right)^{\kappa/2} \mathcal{M}_t^{N}    \big\Vert _{\LL^{p}(\mathbb{R}^{d})}^{2} \Big]\leq
\frac{C}{N^{  1- 2\beta( \kappa +d -d/p + \delta )    }} . 
\]
\end{lemma}

\begin{proof}
The proof follows the previous lemma line by line.
\end{proof}

%%%%%%%%%%%%%%%%%%%%%%%%
%\section*{Acknowledgements}
%%%%%%%%%%%%%%%%%%%%%%%%

%This project is partially supported by the FAPESP-ANR grant SDAIM Stochastic and Deterministic Analysis for Irregular Models $2022/03379-0$ (ANR-22-CE40-0015-01). Author Marielle Simon receives support from the ANR grant MICMOV (ANR-19-CE40-0012) and author Christian Olivera is partially supported by FAPESP by the grant  $2020/04426-6$, and  CNPq by the grant $422145/2023-8$.

\bibliographystyle{ams}

\end{document}